\documentclass[12pt,a4paper]{amsart}

\usepackage{amsmath,amscd,amssymb}
\usepackage[utf8]{inputenc}
\usepackage{times,mathptmx}
\usepackage{tikz}
\usepackage{latexsym}
\usepackage{color}
\usepackage{cite}
\usepackage{hyperref}
\usepackage[initials, bibtex-style]{amsrefs}

\makeatletter

\allowdisplaybreaks[1]

\IfFileExists{graphs/alpha0-1.dat}{%
  \newcommand{\graphpath}[1]{graphs/##1}
}{%
  \newcommand{\graphpath}[1]{##1}%
  \IfFileExists{graphs.tex}{%

  }{}
}

\definecolor{verydarkgreen}{rgb}{0.4,0.47,0.42}
\definecolor{darkblue}{rgb}{0.23,0.33,0.95}
\definecolor{darkred}{rgb}{0.67,0.03,0.06}
\definecolor{darkgreen}{RGB}{55,160,80}
\definecolor{darkorange}{RGB}{230,200,30}
\definecolor{lambda31}{RGB}{30,230,220}

\newtheorem{theorem}{Theorem}
\newtheorem{lemma}[theorem]{Lemma}

\newtheorem{proposition}[theorem]{Proposition}

\theoremstyle{remark}
\newtheorem{remark}[theorem]{Remark}

\newcommand{\IR}{{\mathbb R}}
\newcommand{\IC}{{\mathbb C}}
\newcommand{\IN}{\mathbb N}
\newcommand{\IZ}{\mathbb Z}

\newcommand{\IS}{{\mathbb S}}
\newcommand{\C}{{\mathcal C}}

\newcommand{\Y}{{\mathcal Y}}
\newcommand{\intd}{\,\mathrm{d}}
\DeclareMathOperator{\e}{e}
\renewcommand{\i}{{\mathbf{i}}}
\newcommand{\abs}[1]{\mathopen|#1\mathclose|}

\newcommand{\intervaloo}[1]{\mathopen]#1\mathclose[}
\newcommand{\intervalco}[1]{\mathopen[#1\mathclose[}
\newcommand{\intervaloc}[1]{\mathopen]#1\mathclose]}
\newcommand{\intervalcc}[1]{\mathopen[#1\mathclose]}

\@ifundefined{leqslant}{}{\let\le=\leqslant  \let\leq=\leqslant}
\@ifundefined{geqslant}{}{\let\ge=\geqslant  \let\geq=\geqslant}

\DeclareMathOperator{\sign}{sign}
\let\phi=\varphi

\definecolor{fixme}{RGB}{224,15,182}
\definecolor{Chris}{RGB}{180,180,50}

\@ifundefined{hypersetup}{}{%
  \hypersetup{%
    pdfauthor={C. De Coster \& S. Nicaise \& C. Troestler},
    pdftitle={Buckling spectrum},
    colorlinks=true,
    urlcolor=verydarkgreen,
    citecolor=darkblue,
    linkcolor=darkred,
    pdfstartview=FitH,
  }}

\makeatother

\begin{document}

\title{Spectral analysis of a generalized buckling problem on a ball}

\author{Colette De Coster}
\address{%
  Universit\'e de Valenciennes et du Hainaut Cambr\'esis,
  LAMAV,  FR CNRS 2956,
  Institut des Scien\-ces et Techniques de Valenciennes,
  F-59313  Valenciennes Cedex 9, France
}
\email{Colette.DeCoster@univ-valenciennes.fr}

\author{Serge Nicaise}
\address{%
  Universit\'e de Valenciennes et du Hainaut Cambr\'esis,
  LAMAV,  FR CNRS 2956,
  Institut des Scien\-ces et Techniques de Valenciennes,
  F-59313  Valenciennes Cedex 9, France
}
\email{Serge.Nicaise@univ-valenciennes.fr}

\author{Christophe Troestler}
\thanks{%
    C.~Troestler was partially supported by
    the the project ``Existence and asymptotic behavior of solutions
    to systems of semilinear elliptic partial differential equations''
    (T.1110.14)
    of the F.R.S.-FNRS, \emph{Fonds de la Recherche Fondamentale Collective},
    Belgium. }
\address{%
  D\'epartement de Math\'ematique,
  Universit\'e de Mons,
  Place du parc~20,
  B-7000 Mons, Belgium
}
\email{Christophe.Troestler@umons.ac.be}

\maketitle

\begin{abstract}
  In this paper, the spectrum of the following fourth order problem
  \begin{equation*}
    \begin{cases}
      \Delta^2 u+\nu u=-\lambda \Delta u &\text{in } D_1,\\
      u=\partial_r u= 0 &\text{on } \partial D_1,
    \end{cases}
  \end{equation*}
  where $D_1$ is the unit ball in $\IR^N$, is determined for
  $\nu < 0$ as well as the nodal properties of the corresponding eigenfunctions.
  In particular, we show that the first eigenvalue is simple and that
  the corresponding eigenfunction is radial and
  (up to a multiplicative factor) positive and decreasing with respect to the
  radius.
  This completes earlier results obtained for $\nu \ge 0$
  (see~\cite{DNT1}) and for $\nu <0$
  (see~\cite{LW:13}).
\end{abstract}

\section{Introduction}

In this paper we look for any pairs of real numbers $(\nu,\lambda)$
such that a nontrivial solution $u$ of
\begin{equation}
  \label{eq:pbmgeneral}
  \begin{cases}
    \Delta^2 u+\nu u=-\lambda \Delta u &\text{in } D_1,\\
    u=\partial_r u= 0 &\text{on } \partial D_1,
  \end{cases}
\end{equation}
exists, where $D_1$ is the unit ball of~$\IR^N, N\geq 2$.  This problem
can be viewed as an eigenvalue problem in $\lambda$, once $\nu$ is
fixed, or as an eigenvalue problem in $\nu$, once $\lambda$
is fixed; both cases being points of view on finding
pairs $(\nu,\lambda)$ in the
plane.

There are few references concerning the general problem
\eqref{eq:pbmgeneral} and they concern either the eigenvalue problem
in $\lambda$ ($\ge 0$) for $\nu\geq0$ fixed or the eigenvalue problem
in $\nu$ ($\le 0$)
for $\lambda\le 0$ fixed i.e., only the situation in the first or the
third quadrant of the plane is considered (see in particular
\cites{dG79, Pa56, KLV} and the references therein). These papers are
mainly concerned with the behaviour of functions
$\lambda: \IR^+\to\IR: \nu\mapsto \lambda(\nu)$
(resp. $\nu: \IR^+\to\IR: \lambda\mapsto -\nu(-\lambda)$).
In particular, in \cite{KLV}, the authors prove that the first
eigenvalue $\lambda_1 : \IR^+ \to \IR^+ : \nu
\mapsto \lambda_1(\nu)$, as a function of $\nu$, is strictly concave, increasing, and satisfies
\begin{equation*}
  \forall \nu \in \intervaloo{0, +\infty},\qquad
  \max\bigl\{\lambda_1(0),2\sqrt{\nu} \bigr\}
  < \lambda_1(\nu)
  < \lambda_1(0) + \frac{\nu}{\xi_1}
\end{equation*}
with $\xi_1$ being the first eigenvalue of the Laplacien in $H^1_0(D_1)$ with
corresponding eigenfunction
$\psi_1$.  Similar results are obtained for
$-\nu_1(-\lambda)$. Moreover, they prove that
$-\nu_1(-\lambda)/\lambda\to \xi_1$ as $\lambda$ tends to zero and
that the corresponding eigenfunction converges to $\psi_1$.

More recently, we find in the literature results on the structure of
the eigenfunctions and in particular the positivity of the first
eigenfunction.
There is a vast literature on this last question about the case
$\nu=0$ or $\lambda=0$, see \cites{BDJM, Coffman, CD, CDS, Duffin, GGS,
  GS, GS2, KKM, Sweers, Wieners} which corresponds to the
situations on the axes.

In \cite{LW:13}, the authors consider the problem
\begin{equation}
  \label{eq:pbmgeneral*}
  \begin{cases}
    \Delta^2 u-\tau \Delta u=\omega u  &\text{in } D_1,\\
    u=\partial_r u= 0 &\text{on } \partial D_1,
  \end{cases}
\end{equation}
where the term $ \Delta^2 u$ accounts for the bending, while the term
$-\tau \Delta u$, with $\tau>0$ for stretching; $\omega$ being the
(positive) eigenvalues they are looking for. Observe that reversing
the point of view, problem \eqref{eq:pbmgeneral*} is related to
\eqref{eq:pbmgeneral} with $\nu=-\omega$ and $\lambda=-\tau$.  In
particular, for $\tau\geq 0$, the authors of \cite{LW:13} prove the
existence of $\omega_1$ such that the problem \eqref{eq:pbmgeneral*}
has a positive radially symmetric eigenfunction and that $\omega_1$ is the only
such eigenvalue.  However, the authors observe that
there could exist
non-radially symmetric positive eigenfunctions i.e.,  $\omega_1$ may not be the only eigenvalue with positive eigenfunctions nor the smallest eigenvalue.
Our result proves that, in fact,  it is not the case.
We can also refer to
\cites{Chasman-these,Chasman} where other boundary conditions
are considered.

In \cite{DNT1}, motivated by the study of clamped thin elastic
membranes supported on a fluid substrate, we considered the case
$\nu\geq 0$ and, in particular, we gave a complete description of the
 smallest eigenvalue $\lambda_1$ and its eigenfunction
$u$. In this work, we want to extend the analysis started in
\cite{DNT1} to any $\nu\in \IR$. More precisely, we determine
pairs $(\nu,\lambda)$ with $\nu<0$ such that
\eqref{eq:pbmgeneral} has a nontrivial solution $u$ as well as the
shape of the corresponding solution $u$.  In particular, we obtain
precise information about the smallest eigenvalue $\lambda_1(\nu)$
and its associated eigenfunction as a function of $\nu$.  Putting
together the results of this paper with those of \cite{DNT1}, we
obtain the following theorem concerning the
case where $D_1$ is the unit ball of $\IR^2$.
In this result and throughout the paper, $(j_{k,\ell})_{\ell\ge 1}$
denote the roots of $J_k$, the Bessel function of the first kind of
order $k$.

\begin{theorem}
  \label{lambda1nu}
  If $D$ is the unit ball of $\IR^2$,
  the first eigenvalue $\lambda_1 : \IR \to \IR :
  \nu \mapsto \lambda_1(\nu)$  of~\eqref{eq:pbmgeneral}
  is a continuous
  increasing function of $\nu$ such that
  \begin{equation*}
    \lim_{\nu\to\pm\infty}\lambda_1(\nu)=\pm \infty
    \quad\text{and}\quad
    \lambda_1(0)=j_{1,1}^2.
  \end{equation*}
  Hence it is a bijection from $\IR$ into itself.
  Moreover,
  \begin{itemize}
  \item If $\nu\in \intervaloo{-\infty,(j_{0,1}j_{0,2})^2}$, the first
    eigenvalue is simple and the eigenfunctions $\varphi_1$ are
    radial, one-signed and $\abs{\varphi_1}$ is decreasing with
    respect to the radius $r$.

  \item If
    $\nu\in \intervaloo{(j_{1,n}j_{1,n+1})^2, \,(j_{0,n+1}j_{0,n+2})^2}$,
    for some $n\ge 1$, the first eigenvalue is simple and the
    eigenfunctions are radial and have $n+1$ nodal regions.

  \item If
    $\nu \in \intervaloo{(j_{0,n+1}j_{0,n+2})^2,
      \,(j_{1,n+1}j_{1,n+2})^2}$,
    for some $n\ge 0$, the eigenfunctions $\varphi_1$ have the form
    \begin{equation*}
      R_{1,1}(r) (c_1\cos\theta + c_2\sin\theta),
      \qquad
      c_1, c_2 \in \IR.
    \end{equation*}
    Moreover the function $R_{1,1}$ has $n$ simple zeros
    in $\intervaloo{0,1}$,
    i.e., $\varphi_1$ has $2(n+1)$ nodal regions.
  \end{itemize}
\end{theorem}

Information on the eigenspaces at the countably many $\nu > 0$ not
considered in the previous theorem is also provided in \cite{DNT1}.
For these $\nu$, the eigenspaces have even larger dimensions (see
\cite[Theorem~4.18]{DNT1}).

For $\nu<0$, we can also give a  characterization of higher eigenvalues.
In particular, the nodal properties of their eigenfunctions are
completely determined.

\begin{theorem}
  \label{thm:lambda_kl}
  If $D$ is the unit ball of $\IR^2$,
  there exist increasing differentiable functions
  $\lambda_{k,\ell} : \intervaloo{-\infty, 0} \to \IR : \nu \mapsto
  \lambda_{k,\ell}(\nu)$
  for $k \in \IN$ and $\ell \in \IN^*$ such that the spectrum
  of~\eqref{eq:pbmgeneral} for a given
  $\nu \in \intervaloo{-\infty, 0}$ is exactly
  $\bigl\{ \lambda_{k,\ell}(\nu) \bigm| k \in \IN,\ \ell \in \IN^*
  \bigr\}$.  Moreover
  \begin{equation*}
    \lim_{\nu \to -\infty} \lambda_{k,\ell}(\nu)
    = -\infty
    \quad\text{and}\quad
    \lim_{\nu \to 0} \lambda_{k,\ell}(\nu)
    = j_{k+1,\ell}^2.
  \end{equation*}
  Associated to $\lambda_{k,\ell}$ is a space of eigenfunctions in
  spherical coordinates of the form 
  $R_{k,\ell}(r) \bigl(c_1\cos(k\theta)+ c_2\sin(k\theta)\bigr)$
  where $c_1$, $c_2\in \IR$ and
  \begin{equation*}
    R_{k,\ell}(r)
    := c J_k(\alpha_{k,\ell} \, r)
    + d I_k\Bigl(\frac{\kappa}{\alpha_{k,\ell}} \, r\Bigr),
  \end{equation*}
  for some $(c,d) \ne (0,0)$ suitably chosen (depending on $\kappa$,
  $k$, and $\ell$).
  Here $I_k$ (resp.\ $J_k$) denotes the modified Bessel function
  (resp.\ the Bessel function) of the first kind of order $k$.
  In addition $R_{k,\ell}$ possesses
  $\ell - 1$ roots in $\intervaloo{0,1}$, all of which are simple.
\end{theorem}

Figure~\ref{fig:curves-couples} shows the graph of a few of the
functions $\lambda_{k,\ell}$.  It shows (and we prove) that
\begin{equation*}
  \lambda_1(\nu)
  = \min\bigl\{\lambda_{0,1}(\nu), \lambda_{1,1}(\nu) \bigr\}
\end{equation*}
where the minimum coincides with $\lambda_{0,1}$ for
$\nu \le (j_{0,1}j_{0,2})^2$ and alternates between $\lambda_{0,1}$
and $\lambda_{1,1}$ for larger $\nu$, which explains the results of
Theorem~\ref{lambda1nu}.
The nodal properties of the eigenfunctions are illustrated by the
graphs of the first six eigenfunctions for $\nu =  -1$
which are drawn in Figure~\ref{fig:eigenfunctions}.

\begin{figure}[hbt]
  \centering
  \newcommand{\numin}{-3200}%
  \newcommand{\numax}{5500}%
  \begin{tikzpicture}[x=0.04pt, y=0.6pt]
    \draw[->] (\numin, 0) -- (\numax, 0) node[above]{$\nu$};
    \draw[->] (0, -150) -- (0, 200) node[left]{$\lambda$};
    \draw[fill, color=gray!40] (-1990.587456, 75.1003386) circle(2pt);

    \foreach \v/\y/\t in {
      176.221/36.2544/{j_{0,1}^2 j_{0,2}^2},
      722.624/63.9004/{j_{1,1}^2 j_{1,2}^2},
      2281.9/105.358/{j_{0,2}^2 j_{0,3}^2},
      5094.08/152.718/{j_{1,2}^2 j_{1,3}^2}
    }{
      \draw[dashed, color=gray!90] (\v, \y) -- (\v, -3pt);
      \path (\v, 0) ++(1.4ex, -2.8ex) node[rotate=-50]{$\scriptstyle \t$};
    }
    \draw[color=gray!60, dashed] plot file{\graphpath{lower-bound.dat}};
    \node[color=gray!60, below left, yshift=3pt] at (4500, 130)
    {\rotatebox{15}{$\scriptstyle \lambda = 2\sqrt{\nu}$}};
    \begin{scope}
      \clip (\numin, -150) rectangle (\numax, 200);
      \foreach \k/\l/\c/\xy in {
        3/1/lambda31/{-2000, 10}, 
        0/1/darkred/{-300,-80},  0/2/darkred/{-2950, -26},
        0/3/darkred/{-500, 108},
        1/1/darkblue/{-1000, -60},  1/2/darkblue/{-2950, 27},
        1/3/darkblue/{-600, 142},
        2/1/darkorange/{-1800, -40},
        2/2/darkorange/{-2400, 48},
        5/1/darkgreen/{-300, 80}}{
        \draw[color=\c] plot file{\graphpath{eigenval-\k-\l.dat}};
        \node[color=\c] at (\xy) {$\lambda_{\k,\l}$};
      }
    \end{scope}
  \end{tikzpicture}
  \caption{Curves of $(\nu, \lambda)$ along which
    \eqref{eq:pbmgeneral} possesses a nontrivial solution.
    The dot indicates where the graphs of $\lambda_{0,3}$ (red)
    and $\lambda_{5,1}$ (green) cross.}
  \label{fig:curves-couples}
\end{figure}
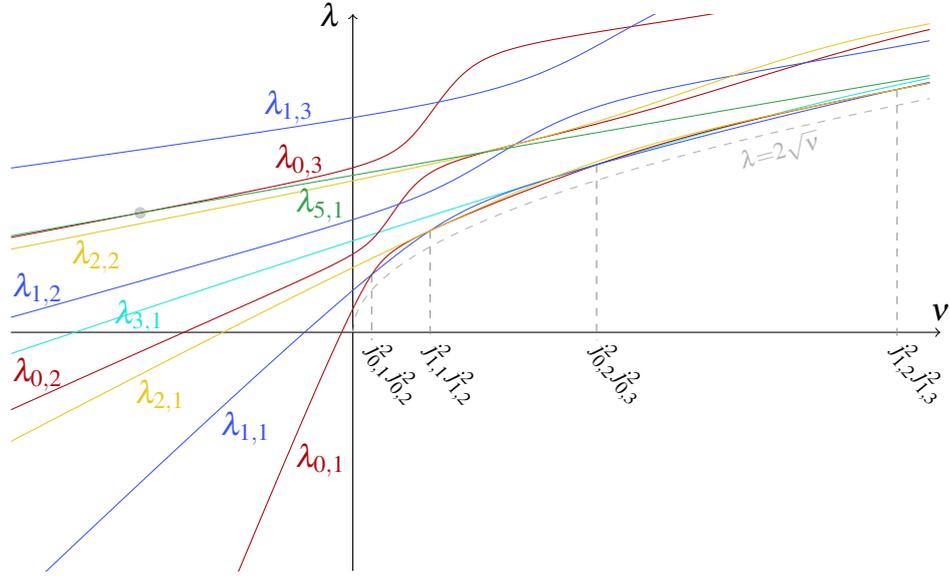

\begin{figure}[bht]
  \centering
  \begin{minipage}[b]{0.31\linewidth}
    \centering
    \includegraphics[width=\linewidth, viewport=56 43 571 403]{%
      \graphpath{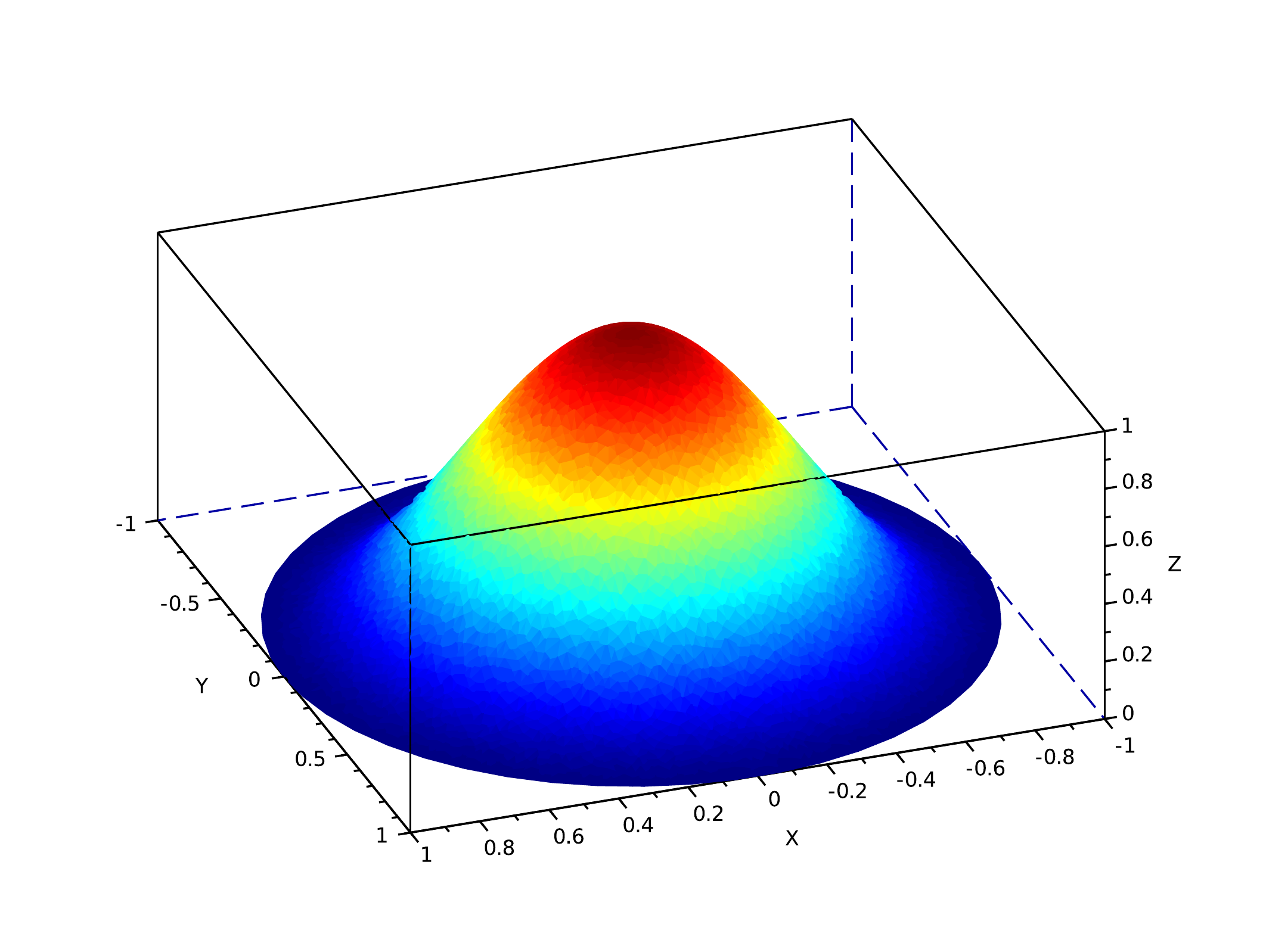}}
    $\phi_1$,\hspace{0.5em} $(k,\ell) = (0,1)$
  \end{minipage}
  \hfill
    \begin{minipage}[b]{0.31\linewidth}
    \centering
    \includegraphics[width=\linewidth, viewport=56 39 574 403]{%
      \graphpath{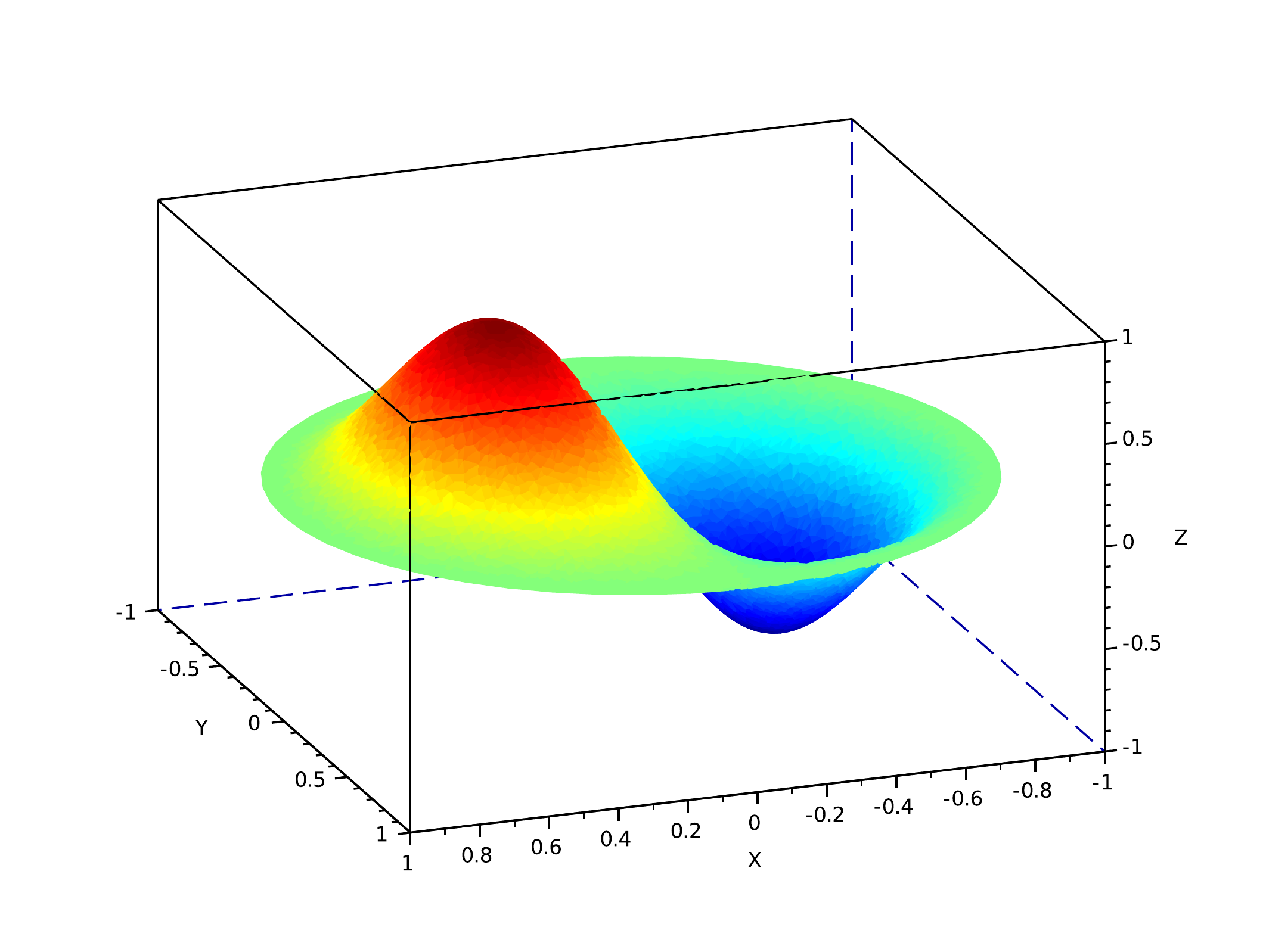}}
    $\phi_2$,\hspace{0.5em} $(k,\ell) = (1,1)$
  \end{minipage}
  \hfill
    \begin{minipage}[b]{0.31\linewidth}
    \centering
    \includegraphics[width=\linewidth, viewport=56 43 574 403]{%
      \graphpath{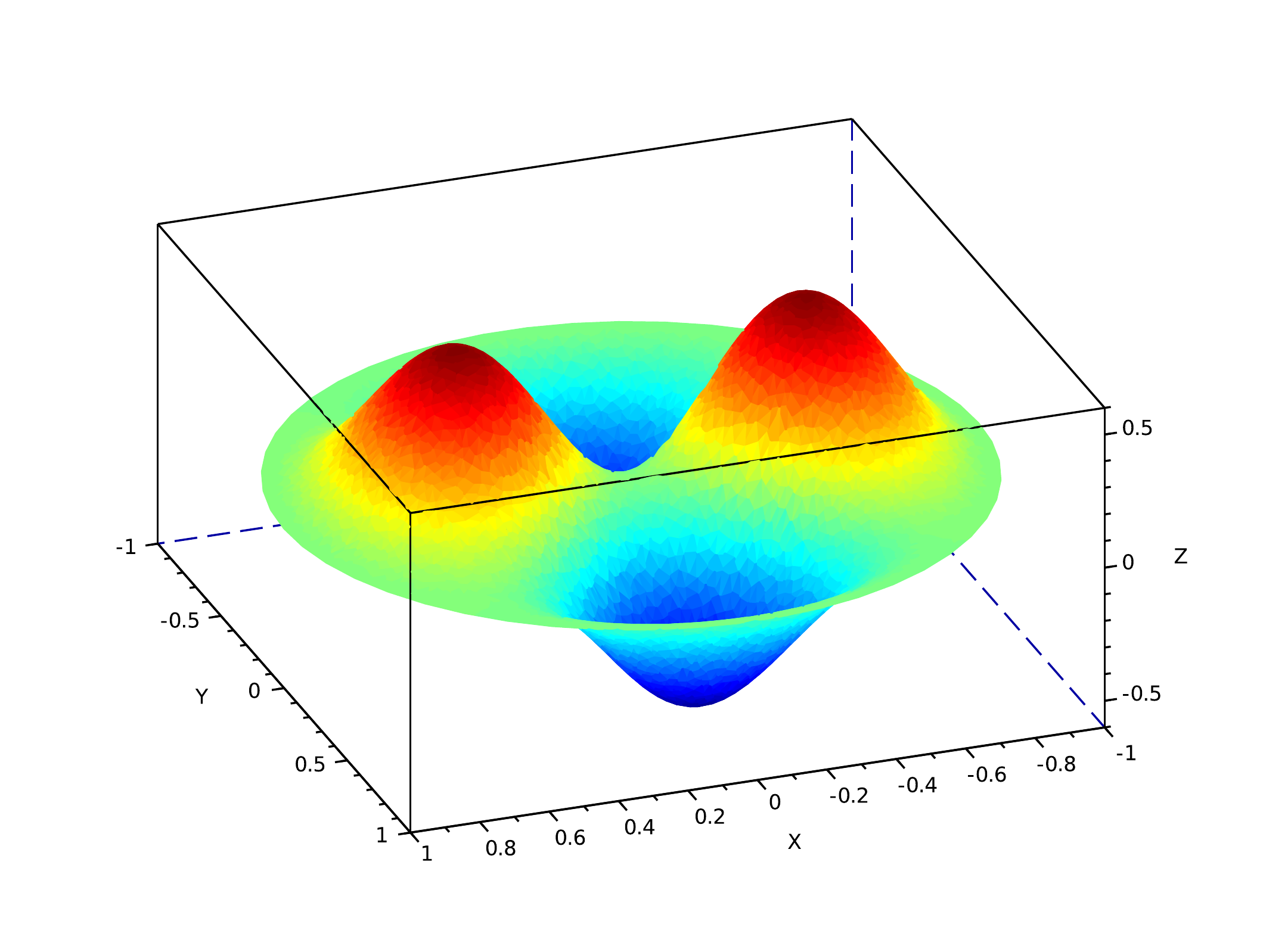}}
    $\phi_3$,\hspace{0.5em} $(k,\ell) = (2,1)$
  \end{minipage}

  \vspace{2ex}
  \begin{minipage}[b]{0.31\linewidth}
    \centering
    \includegraphics[width=\linewidth, viewport=56 39 574 403]{%
      \graphpath{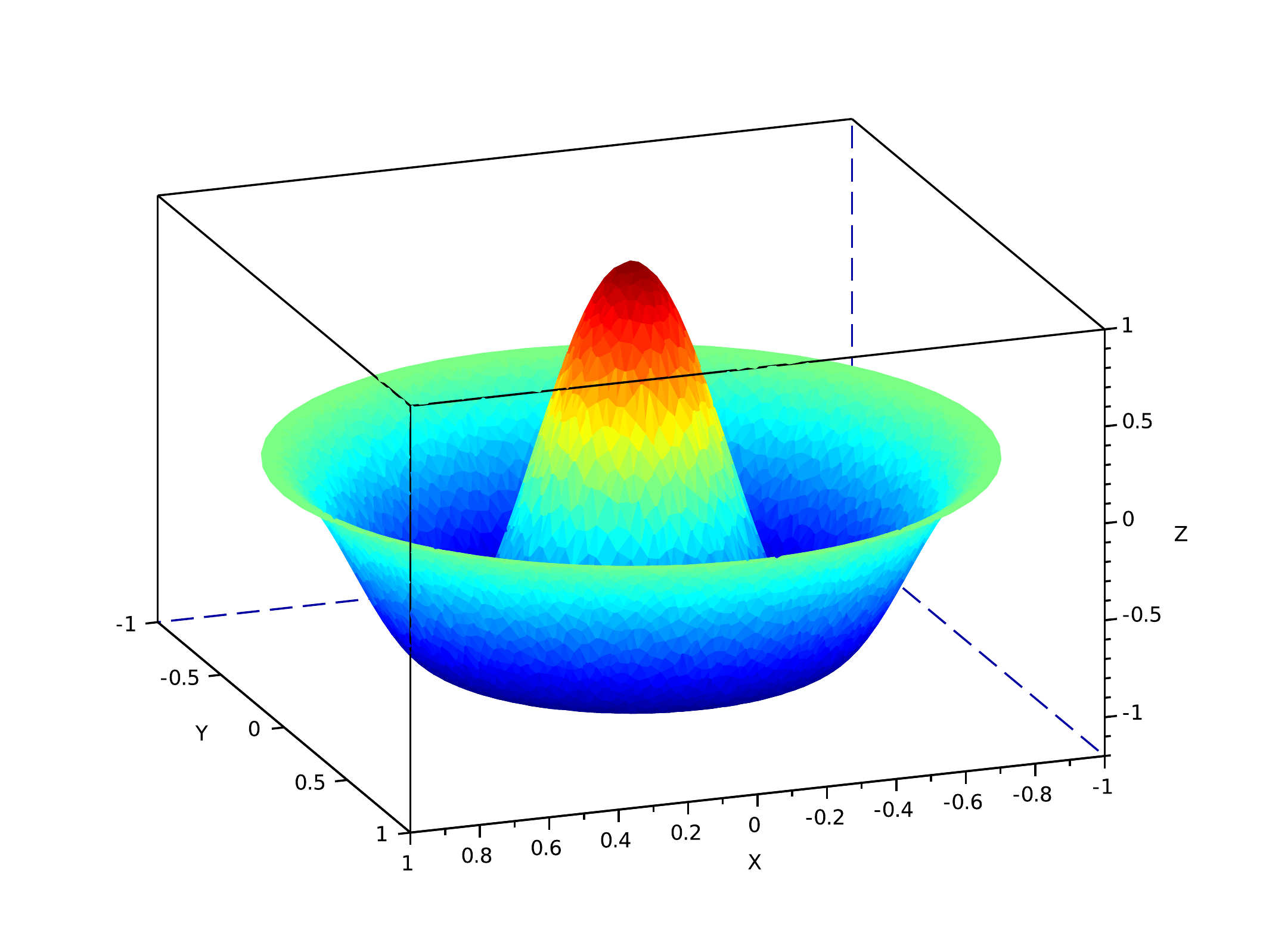}}
    $\phi_4$,\hspace{0.5em} $(k,\ell) = (0,2)$
  \end{minipage}
  \hfill
    \begin{minipage}[b]{0.31\linewidth}
    \centering
    \includegraphics[width=\linewidth, viewport=56 41 575 403]{%
      \graphpath{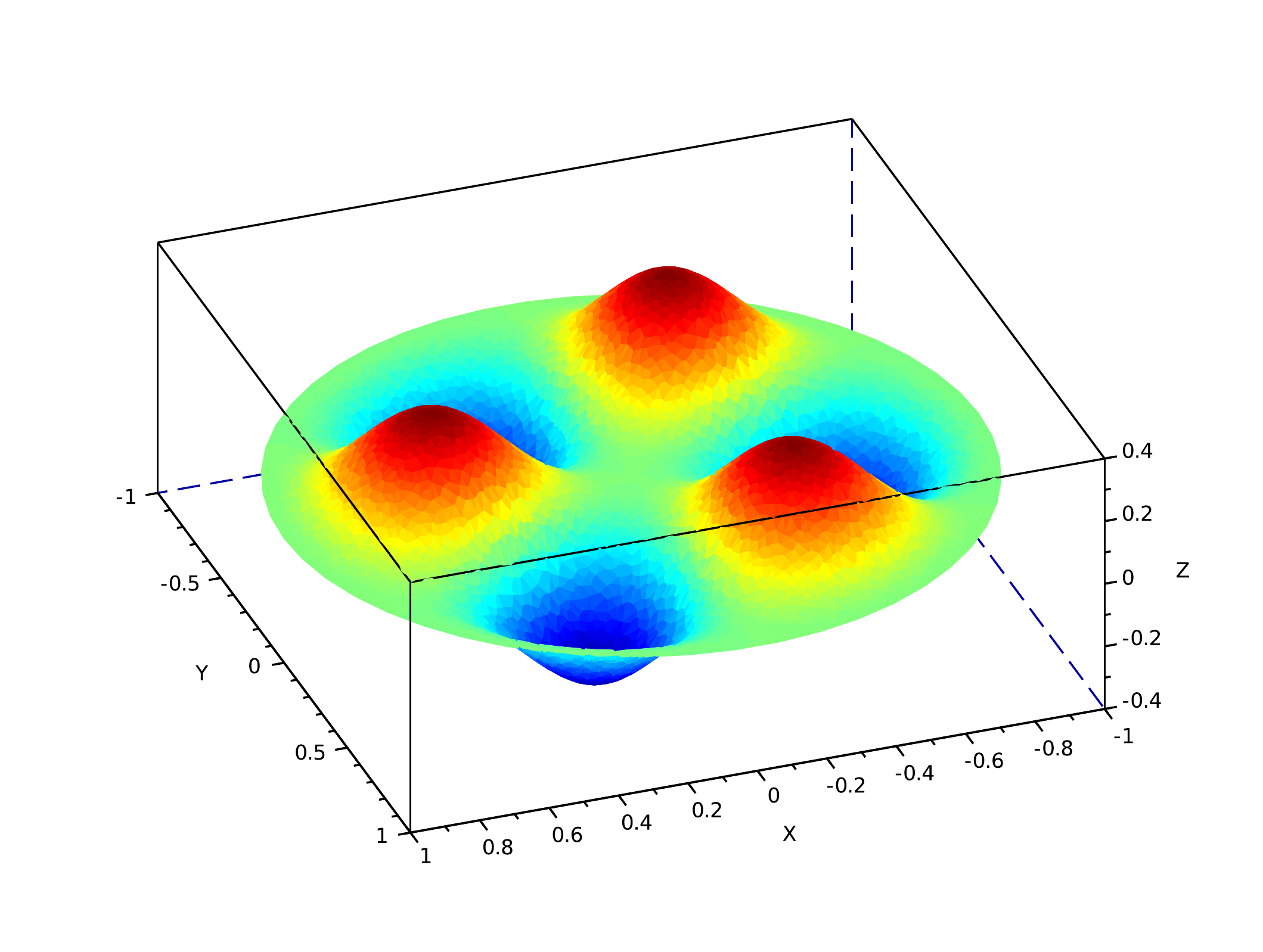}}
    $\phi_5$,\hspace{0.5em} $(k,\ell) = (3,1)$
  \end{minipage}
  \hfill
    \begin{minipage}[b]{0.31\linewidth}
    \centering
    \includegraphics[width=\linewidth, viewport=56 43 575 403]{%
      \graphpath{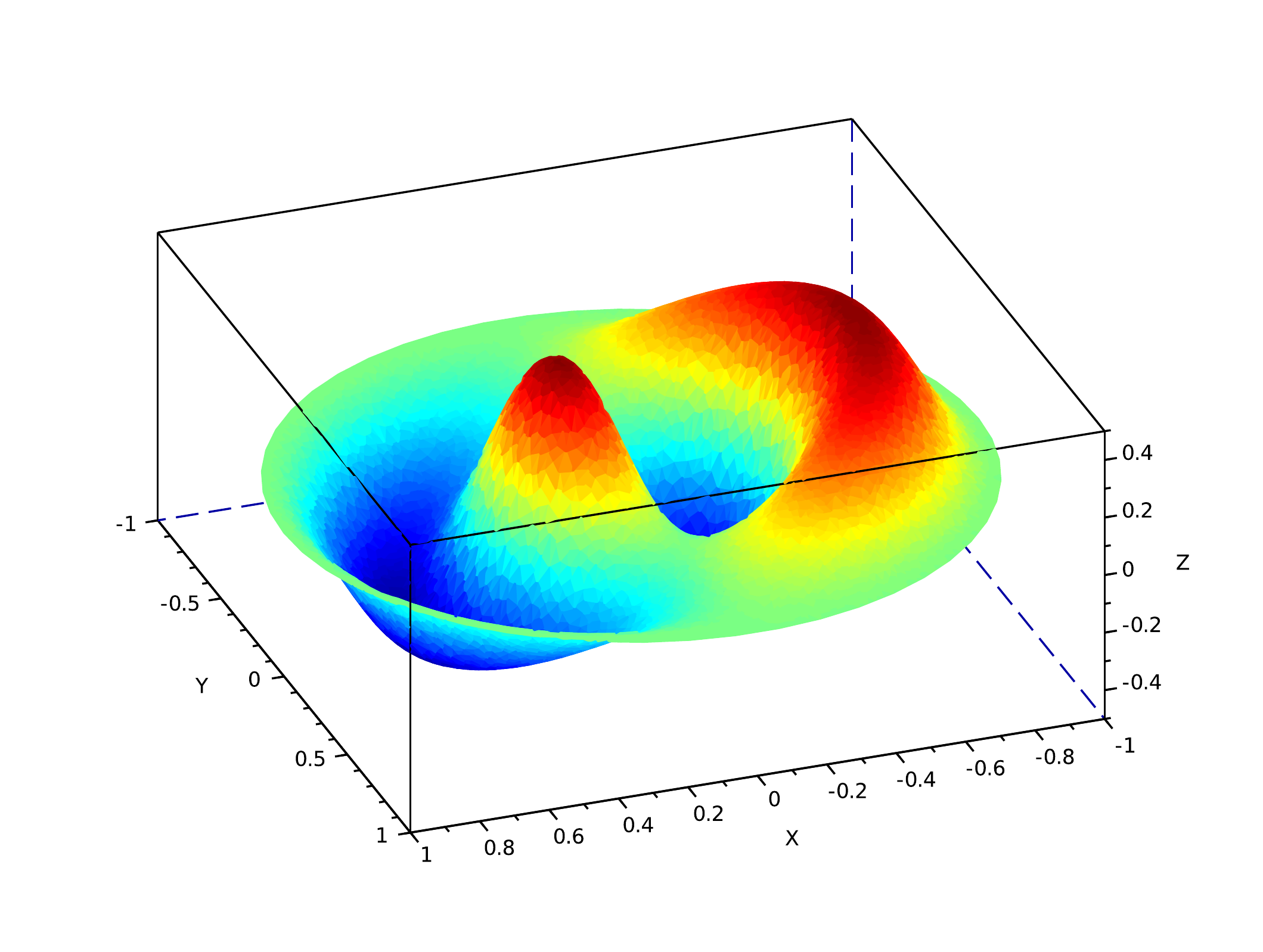}}
    $\phi_6$,\hspace{0.5em} $(k,\ell) = (1,2)$
  \end{minipage}

  \caption{Graphs of the first eigenfunctions for $\nu=- 1$,
    mentioning the value $(k,\ell)$ of Theorem~\ref{thm:lambda_kl} to
    which they correspond.}
  \label{fig:eigenfunctions}
\end{figure}

\medbreak

The paper is organized as follows and concerns the case $\nu<0$.
 In Section~\ref{Sect1}, we explain how to find solutions
to~\eqref{eq:pbmgeneral} despite the fact that the method of
separation of variables is not directly applicable. In
Section~\ref{Sect2}, we show (see Theorem~\ref{thm:alpha kl}) that,
for all $k \in \IN$, there exists an increasing sequence
$\alpha_{k,\ell} = \alpha_{k,\ell}(\kappa) \in \intervaloo{
  j_{k,\ell},\, j_{k+1,\ell}}$,
$\ell \ge 1$, with $\kappa=\sqrt{-\nu}$, such that
$\lambda_{k,\ell}(\nu) = \alpha_{k,\ell}^2 -
{\kappa^2}/{\alpha_{k,\ell}^2}$
is an eigenvalue of~\eqref{eq:pbmgeneral} with corresponding
eigenfunctions of the form $R_{k,\ell}(r)\e^{\pm\i k\theta}$.
The minimal value of the spectrum
$\lambda_1 = \lambda_1(\kappa) := \min_{k,\ell} \lambda_{k,\ell}$
corresponds to the  minimum of
$\{ \alpha_{k,\ell} \mid k \in \IN,\ \ell \ge 1 \}$
which is given by $\alpha_{0,1}$.


In Section~\ref{sec:nodal prop}, we show that $\lambda_1(\nu)$ is
simple and its eigenfunction $\varphi_1$ is radial, one-signed and
$\abs{\varphi_1}$ is decreasing with respect to $r$ (see
Theorem~\ref{R01}).
We further give precise statements about the nodal properties of the
other eigenfunctions in Theorem~\ref{thm:Rkl nodal}.

Theorem \ref{lambda1nu} can then be deduced from Theorem
\ref{thm:lambda1} together with \cite[Theorem 4.3, Lemmas 4.4, 4.5 and
4.6]{DNT1} for what concerns the behaviour of the first eigenvalue
while the information concerning the first eigenfunction are deduced
from Theorem \ref{R01} together with  \cite[Theorems~4.17, 4.18, 5.5
and Proposition~5.11]{DNT1}.
Theorem~\ref{thm:lambda_kl} is a direct consequence from
Theorem~\ref{thm:alpha kl}, Lemma~\ref{alpha decrease},
Lemma~\ref{kappa->0}, Theorem~\ref{thm:Rk1>0} and Theorem~\ref{thm:Rkl
  nodal}.

The case $D_1\subseteq \IR^N$ with $N>2$ is also considered in
section~\ref{sec:N>2}.

\section{Preliminaries}
\label{Sect1}

Let us first consider the case $N=2$ (see Section~\ref{sec:N>2} for
the case $N>2$).
\medbreak

As the case $\nu\geq 0$ in \eqref{eq:pbmgeneral} was
fully treated in \cite{DNT1}, we can restrict ourselves to the case
$\nu< 0$, that we rewrite as
\begin{equation}
  \label{eq:pbm}
  \begin{cases}
    \Delta^2 u- \kappa^2 u=-\lambda \Delta u &\text{in } D_1,\\
    u=\partial_r u= 0 &\text{on } \partial D_1,
  \end{cases}
\end{equation}
with $\kappa> 0$.
Similarly to \cite{DNT1}, we factorize
\begin{equation}\label{serge3}
  \Delta^2u+\lambda \Delta u -\kappa^2 \, u
  = (\Delta+\alpha^2) (\Delta+\beta^2)u
  = 0
\end{equation}
with $\alpha, \beta\in \IC$ such that
$\alpha^2 \beta^2 = -\kappa^2$ and $\alpha^2+\beta^2 = \lambda$. As
$\kappa>0$, $\alpha$ and $\beta$ are both non zero and we can choose
$\alpha$ as a positive real number and $\beta={\i\kappa}/{\alpha}$
that is clearly different from $\alpha$.

As we work in two dimensions, we use the ansatz
$u(r,\theta)=R(r) \e^{\i k\theta}$ with $k\in \IZ$, where $(r,\theta)$
are the polar coordinates, and notice that~\eqref{serge3} is
equivalent to a fourth order ordinary differential equation (in $\partial_r$)
\begin{equation}\label{serge1polar}
  L(\partial_r, r, \alpha,\beta, \abs{k}) R=0.
\end{equation}
Hence by the theory of ordinary differential equations, $L$ has four
linearly independent solutions.  To find them it suffices to notice
that
\begin{equation*}
  (\Delta+\alpha^2)u=0
  \quad\Rightarrow\quad
  (\Delta+\alpha^2)(\Delta+\beta^2)u=0.
\end{equation*}
Thus, if
\begin{equation}
  \label{serge4}
  (\Delta+\alpha^2)\bigl(R(r) \e^{\i k\theta}\bigr)=0,
\end{equation}
then $R$ is a solution to~\eqref{serge1polar}.  But a solution
to~\eqref{serge4} is simpler to find. Indeed such $R$ satisfy the
Bessel equation
\begin{equation*}
  (r\partial_r)^2 R+\alpha^2 r^2 R=k^2 R.
\end{equation*}
As $\alpha \ne 0$, $R$ is then a linear combination of
$J_{|k|}(\alpha r)$ and of $Y_{|k|}(\alpha r)$.

Similarly, if
\begin{equation}
  \label{serge4bis}
  \Bigl(\Delta-\frac{\kappa^2}{\alpha^2} \Bigr)
  \bigl(R(r) \e^{\i k\theta}\bigr)=0,
\end{equation}
then $R$ is a solution to~\eqref{serge1polar}.  But $R$ is solution
to~\eqref{serge4bis} if and only if $R$ satisfies the modified Bessel
equation
\begin{equation*}
  (r\partial_r)^2 R-\frac{\kappa^2}{\alpha^2} r^2 R=k^2 R.
\end{equation*}
Therefore, $R$ is a linear combination of
$I_{|k|}(\frac{\kappa r}{\alpha})$ and of
$K_{|k|}(\frac{\kappa r}{\alpha})$.

\medbreak

Summing up, we have proved the following result.

\begin{lemma}
  \label{lserge1}
  Let $k\in \IZ$.  If $\alpha\in \intervaloo{0,\infty}$ and
  $\beta = \i\kappa/\alpha$, then the four linearly independent
  solutions to~\eqref{serge1polar} are $J_{|k|}(\alpha r)$,
  $Y_{|k|}(\alpha r)$, $I_{|k|}(\frac{\kappa r}{\alpha})$ and
  $K_{|k|}(\frac{\kappa r}{\alpha})$.
\end{lemma}

\section{Eigenvalue problem}
\label{Sect2}

Here we want to characterize the full spectrum of the buckling problem
on the unit disk.  In other words, we look for a
$u \ne 0$ and $\lambda \in \IR$ such that
\begin{equation}
  \label{serge5}
  \begin{cases}
    \Delta^2 u- \kappa^2 u=-\lambda \Delta u &\text{ in } D_1,\\
    u=\partial_r u= 0 &\text{ on } \partial D_1,
  \end{cases}
\end{equation}
where $D_1$ is the unit ball of $\mathbb R^2$.

\begin{proposition}
  \label{eigenfunctions}
  The eigenfunctions of the boundary value problem~\eqref{serge5} are
  of the form $u=R(r) \e^{\i k\theta}$ with $k\in \IZ$ and $R$ given by
  \begin{equation}
    \label{eq R}
    R(r) = c J_{|k|}(\alpha r)
    + d I_{|k|}\Bigl(\frac{\kappa}{\alpha} r \Bigr),
  \end{equation}
  for some $c,d\in \IR$,  where  $\alpha$ is a positive solution to
  \begin{equation}\label{Dbis}
    F_k(\alpha) :=
    \frac{\kappa}{\alpha} J_{|k|}(\alpha)
    I'_{|k|} \Bigl(\frac{\kappa}{\alpha} \Bigr)
    - \alpha I_{|k|}\Bigl(\frac{\kappa}{\alpha}\Bigr)  J'_{|k|}(\alpha )
    =0.
  \end{equation}
  The corresponding eigenvalue is $\lambda = \alpha^2 -
  {\kappa^2}/{\alpha^2}$.
\end{proposition}
\begin{proof}
  According to the previous section, we look for solutions $u$ to
  \eqref{eq:pbm} in the form
  \begin{equation*}
    u=R(r) \e^{\i k\theta}, \qquad
    \text{with } k\in \IZ.
  \end{equation*}
  From  Lemma \ref{lserge1}, we see that
  \begin{equation}
    \label{eq Rbis}
    R(r) = c J_{|k|}(\alpha r)
    + d I_{|k|}\Bigl(\frac{\kappa}{\alpha} r \Bigr),
  \end{equation}
  for some $c,d\in \IR$ (since $R$ and $R'$ are bounded near $r=0$).
  Hence the boundary conditions at $r=1$ lead to the system
  \begin{equation}
    \label{eq:R bd cond}
    \begin{cases}
      c\, J_{|k|}(\alpha)
      + d\, I_{|k|}\bigl(\frac{\kappa}{\alpha} \bigr) = 0,
      \\[1\jot]
      c \,\alpha J'_{|k|}(\alpha )
      + d\, \frac{\kappa}{\alpha}
      I'_{|k|}\bigl(\frac{\kappa}{\alpha} \bigr) = 0.
    \end{cases}
  \end{equation}
  This $2\times 2$ system has a non-trivial solution $(c,d)$ if and only if
  its determinant is equal to zero, namely if and only if~\eqref{Dbis}
  is satisfied.

  The fact that $(\e^{\i k \theta})_{k \in \IZ}$ form a basis
  of $L^2(\intervaloo{0, 2\pi})$
  essentially allows to conclude that no other eigenvalues exist
  (see \cite[Theorem~3.2]{DNT1} for details).
\end{proof}

Let us start with a technical result.

\begin{lemma}
  \label{cserge1}%
  Let $k \in \IN$.  Then the function
  \begin{equation*}
    G_k : \intervaloo{0, +\infty} \to \IR :
    z \mapsto \frac{z I'_k(z)}{I_k(z)}
  \end{equation*}
  has a positive derivative and thus is increasing on
  $\intervaloo{0, +\infty}$.
\end{lemma}

\medbreak
\begin{remark}
  \label{rem g_k}
  Observe that, by \eqref{A:4I},
  \begin{equation}
    \label{equ g_k}
    G_k(z) = \frac{z I'_k(z)}{I_k(z)}= \frac{z I_{k+1}(z)}{I_k(z)}+k.
  \end{equation}
  Hence by \cite[Theorem 1.1]{Laforgia-Natalini:10} we have that
  \begin{equation*}
    G_k(z)> -1+\sqrt{(k+1)^2+z^2}.
  \end{equation*}
  This implies in particular that
  \begin{equation*}
    \lim_{z\to+\infty} G_k(z)=+\infty.
  \end{equation*}
\end{remark}

\begin{proof}
  Direct calculations using the differential equation satisfied by
  modified Bessel functions \eqref{A:6I} yield
  \begin{equation*}
    G_k'(z) = \frac{(z^2+k^2) I_k^2(z)-z^2 (I'_k(z))^2}{zI_k^2(z)}.
  \end{equation*}
  Then using the recurrence relation \eqref{A:4I}, we deduce that
  \begin{equation*}
    G_k'(z)
    = \frac{z^2(I_k^2(z)-I_{k+1}^2(z))-2k z I_k(z) I_{k+1}(z)}{z I_k^2(z)}.
  \end{equation*}
  Hence with the notation $u=\frac{ I_{k+1}(z)}{ I_{k}(z)}$ from
  Lemma~\ref{lextLaforgia} (in the appendix), we have
  \begin{equation}
    \label{derivee de g}
    G_k'(z) = -z \Bigl( u^2 + \frac{2k}{z}u - 1 \Bigr),
  \end{equation}
  and then Lemma~\ref{lextLaforgia} implies that $G_k'(z)>0$.
\end{proof}

\begin{theorem}
  \label{thm:alpha kl}
  For all $k\in \IN$ and $\kappa>0$, the roots of~$F_k$ (defined
  by~\eqref{Dbis}\/) are simple and can be ordered
  as an increasing sequence
  $\alpha_{k,\ell} = \alpha_{k,\ell}(\kappa) > 0$, with
  $\ell\in \IN^*$, such that
  \begin{gather*}
    \forall \ell > 0,\quad
    j_{k,\ell}<\alpha_{k,\ell}<  j_{k+1,\ell}.
  \end{gather*}
  Each $\ell > 0$ gives rise to the eigenvalue
  \begin{equation}
    \label{rel_lambda_alpha}
    \lambda_{k,\ell} = \alpha_{k,\ell}^2 -
    \frac{\kappa^2}{\alpha_{k,\ell}^2}  \, ,
  \end{equation}
  of~\eqref{serge5}
  and to corresponding eigenfunctions of the form
  $R_{k,\ell}(r)\e^{\pm\i\, k\theta}$ with
  \begin{equation*}
    R_{k,\ell}(r) = cJ_k(\alpha_{k,\ell}\, r)
    + d I_k \Bigl(\frac{\kappa}{\alpha_{k,\ell}} \, r\Bigr),
  \end{equation*}
  where $(c, d)$ is a solution
  to~\eqref{eq:R bd cond} with $\alpha=\alpha_{k,\ell}$.
 \end{theorem}

\begin{proof}
  Since $I_k$ is positive on $\intervaloo{0, +\infty}$ and the
  positive roots of $J_k$ are simple, the positive roots of $F_k$
  never coincide with those of $J_k$.
  Hence, we can write
  \begin{equation*}
    F_k(\alpha) = J_k(\alpha) \, I_k\Bigl(\frac{\kappa}{\alpha}\Bigr)
    \, \tilde F_k(\alpha)
  \end{equation*}
  with
  \begin{equation*}
    \tilde F_k(\alpha)
    := G_k\Bigl(\frac{\kappa}{\alpha}\Bigr)- H_k(\alpha)
  \end{equation*}
  where $G_k$ was defined in Lemma~\ref{cserge1} and
  $H_k(z) := \frac{z J'_k(z)}{J_k(z)}$ (see \cite[Lemma~4.1]{DNT1}) and
  we have
  \begin{equation*}
    F_k(\alpha) =0 \qquad \Leftrightarrow \qquad
    \tilde F_k(\alpha)=0.
  \end{equation*}
  Using formulas~\eqref{A:4} and~\eqref{A:4I}, $\tilde F_k$ may also
  be written as
  \begin{equation}
    \label{eq:Fk-succ}
    \tilde F_k(\alpha)
    = \frac{\kappa}{\alpha} \,
    \frac{I_{k+1}\bigl(\frac{\kappa}{\alpha}\bigr)}{
      I_k\bigl(\frac{\kappa}{\alpha}\bigr)}
    + \alpha \frac{J_{k+1}(\alpha)}{J_k(\alpha)}.
  \end{equation}
  Observe first that if $\alpha \in \intervaloo{0, j_{k,1}}$, then
  $J_k(\alpha) > 0$ and $J_{k+1}(\alpha) > 0$ and so $\tilde
  F_k(\alpha) > 0$.  Thus the roots of $\tilde F_k$ lie in
  $\intervaloo{j_{k,1}, +\infty}$.
  Easy computations show that, for all $\ell \ge 1$,
  \begin{equation*}
    \lim_{\alpha\xrightarrow{>} j_{k,\ell}} \tilde F_k(\alpha)
    = -\infty
    \qquad \text{and} \qquad
    \tilde F_k(j_{k+1,\ell})
    =\frac{\kappa}{j_{k+1,\ell}} \,
    \frac{I_{k+1}\bigl(\frac{\kappa}{j_{k+1,\ell}}\bigr)}{
      I_k\bigl(\frac{\kappa}{j_{k+1,\ell}}\bigr)}
    > 0
  \end{equation*}
  This implies that $\tilde F_k$ possesses a root between $j_{k,\ell}$ and
  $j_{k,\ell+1}$.

  To establish that it has a unique root in this interval,
  it is enough to prove that, for every root $\alpha^*$,
  we have $\tilde F_k'(\alpha^*)>0$.
  Direct calculations yield
  \begin{equation}
    \label{deriveealpha}
    \tilde F_k'(\alpha)
    = - \frac{\kappa}{\alpha^2}
    G'_k\Bigl(\frac{\kappa}{\alpha}\Bigr) - H'_k(\alpha).
  \end{equation}
  First, let us use \cite[Lemma~4.1]{DNT1} which states that
  \begin{equation*}
    H'_k(z) = \frac{k^2-z^2}{z}-z \frac{(J_k'(z))^2}{J_k(z)^2},
  \end{equation*}
  and, using \eqref{A:4}, we obtain
  \begin{equation*}
    H'_k(z)
    = -z+2k \frac{J_{k+1}(z)}{J_k(z)}-z\frac{J_{k+1}^2(z)}{J_k^2(z)}.
  \end{equation*}
  Now using that $\alpha^*$ is a root of $\tilde F_k$ and remembering
  \eqref{eq:Fk-succ}, we have
  \begin{equation*}
    \frac{J_{k+1}(\alpha^*)}{J_k(\alpha^*)}=-\frac{\kappa}{(\alpha^*)^2} u,
  \end{equation*}
  where we have set $u := \frac{I_{k+1}(z)}{I_{k}(z)}$ with
  $z = \frac{\kappa}{\alpha^*}$.  These two identities show that
  \begin{equation*}
    H'_k(\alpha^*)
    = -\alpha^* -2k\frac{\kappa}{(\alpha^*)^2} u
    - \frac{\kappa^2}{(\alpha^*)^3} u^2.
  \end{equation*}
  Using this identity together with \eqref{derivee de g} in
  \eqref{deriveealpha}, we get
  \begin{equation*}
    \begin{aligned}
      \tilde F_k'(\alpha^*)
      &= \frac{\kappa^2}{(\alpha^*)^3}
      \Bigl(u^2 + \frac{2k\alpha^*}{\kappa}u - 1 \Bigr)
      + \alpha^* + 2k\frac{\kappa}{(\alpha^*)^2} u
      + \frac{\kappa^2}{(\alpha^*)^3} u^2\\
      &= \frac{2\kappa^2}{(\alpha^*)^3}
      \Bigl(u^2 + \frac{2k \alpha^*}{\kappa} u
      + \frac{(\alpha^*)^4-\kappa^2}{2\kappa^2} \Bigr)\\
      &= \frac{2\kappa^2}{(\alpha^*)^3}
      \Bigl( u^2 + \frac{2k u}{z} + \frac{(\alpha^*)^2}{2z^2}
      - \frac12 \Bigr)
      \qquad\text{with } z = {\kappa}/{\alpha^*}.
    \end{aligned}
  \end{equation*}
  Hence to show that $\tilde F_k'(\alpha^*)$ is positive, it remains
  to prove that
  \begin{equation}
    \label{colette}
    u^2 + \frac{2k u}{z} + \frac{(\alpha^*)^2}{2z^2} - \frac12
    > 0.
  \end{equation}
  But the estimate (2.4) of \cite{Laforgia-Natalini:10} says that
  \begin{equation*}
    u^2 + \frac{2(k+1)}{z}u - 1 > 0,  
  \end{equation*}
  therefore
  \begin{equation*}
    u^2 +\frac{2k u}{z} + \frac{(\alpha^*)^2}{2z^2} - \frac12
    > \frac12-\frac{2 u}{z}+\frac{(\alpha^*)^2}{2z^2}=
    \frac{1}{2z^2} \bigl(z^2-4uz+(\alpha^*)^2 \bigr).    
  \end{equation*}
  Since $\alpha^* > j_{k,1} > j_{0,1} \ge \frac{3\pi}{4}>2$
  (see \cite[Theorem~3]{Hethcote:69} for the estimate on $j_{0,1}$)
  and, by
  Lemma~\ref{lextLaforgia}, $u<1$ for $k \ge 0$ and $z>0$, we deduce that
  \begin{equation*}
    u^2+\frac{2k u}{z}+\frac{(\alpha^*)^2}{2z^2}-\frac12
    > \frac{1}{2z^2}(z-2)^2
    \ge 0.    
  \end{equation*}
  This implies that, for every root $\alpha^*$, we have
  $\tilde F_k'(\alpha^*)>0$ and prove the uniqueness of the root of
  $\tilde F_k$ in $\intervaloo{j_{k,\ell}, j_{k,\ell+1}}$. This
  concludes the proof.
\end{proof}

We now give some further informations on the functions
$\alpha_{k,\ell}$.

\begin{lemma}
  \label{alpha decrease}%
  For all $k \in \IN$ and $\ell \in \IN^*$, the function
  $\alpha_{k,\ell}: \intervaloo{0, +\infty} \to \IR : \kappa \mapsto
  \alpha_{k,\ell}(\kappa)$
  is of class $\C^1$ and $\partial_\kappa \alpha_{k,\ell} < 0$.
\end{lemma}

\begin{proof}
  Let us note $F_k(\alpha, \kappa)$ the function $F_k(\alpha)$ defined
  by~\eqref{Dbis} where we have explicited the dependence on $\kappa$.
  The assertion will result from the Implicit Function Theorem.  Let
  us fix $k \in \IN$, $\kappa^* > 0$ and
  $\alpha^* = \alpha_{k,\ell}(\kappa^*) > 0$.

  Again as in the proof of Theorem \ref{thm:alpha kl}, instead of
  working on $F_k(\alpha, \kappa)$, we consider the function
  $\tilde F_k(\alpha, \kappa)$.  In the proof of
  Theorem~\ref{thm:alpha kl}, we already observed that
  \begin{equation*}
    \partial_\alpha \tilde F_k(\alpha^*, \kappa^*) > 0.    
  \end{equation*}
  Hence we can apply the Implicit Function Theorem to $\tilde F_k$ and
  there exists a $\C^1$ curve $\beta_{\ell}$ defined around $\kappa^*$
  such that, in a neighbourhood $V$ of $(\alpha^*,\kappa^*)$,
  \begin{equation*}
    F_k(\alpha,\kappa)=0 \text{ if and only if }\alpha
    = \beta_{\ell}(\kappa).
  \end{equation*}
  Moreover, using Lemma~\ref{cserge1}, it is easily seen that
  \begin{equation*}
    \partial_\kappa \tilde F_k(\alpha^*, \kappa^*) > 0
  \end{equation*}
  and so
  \begin{equation*}
    \partial_\kappa \beta_{\ell}(\kappa^*) =
    - \frac{\partial_\kappa \tilde F_k(\alpha^*, \kappa^*)}{
      \mathstrut\partial_\alpha \tilde F_k(\alpha^*, \kappa^*)}
    < 0.
  \end{equation*}
  As, for all $\kappa$, $\alpha_{k,\ell}(\kappa)$ is the only root of
  $\tilde F_k$ in $\intervaloo{j_{k,\ell}, j_{k,\ell+1}}$, we have,
  for all $ \kappa \in V$,
  \begin{equation*}
    \beta_\ell(\kappa) = \alpha_{k,\ell} (\kappa),
  \end{equation*}
  whence the desired result.
\end{proof}

\begin{lemma}
  \label{kappa->0}
  Let $k \in \IN$ and $\ell > 0$. Then we have
  \begin{equation*}
    \lim_{\kappa \to 0}\alpha_{k,\ell}(\kappa) = j_{k+1,\ell}
    \qquad \text{while}\qquad
    \lim_{\kappa \to \infty}\alpha_{k,\ell}(\kappa) = j_{k,\ell}.
  \end{equation*}
\end{lemma}

\begin{proof}
  As $\alpha_{k,\ell}$ is decreasing and bounded
  (for all $\kappa > 0$, $j_{k,\ell} < \alpha_{k,\ell}(\kappa) <
  j_{k+1,\ell}$), these two limits exist. Let
  us denote
  \begin{equation*}
    \alpha_0 := \lim_{\kappa \to 0}\alpha_{k,\ell}(\kappa)
    \qquad \text{and}\qquad
    \alpha_{\infty} := \lim_{\kappa \to \infty}\alpha_{k,\ell}(\kappa).
  \end{equation*}
  It remains to find their value.  Recall that, for all $z>0$, we have
  $0<\frac{I_{k+1}(z)}{I_k(z)}<1$ and hence, formula
  \eqref{eq:Fk-succ} implies
  \begin{align*}
    0 &= \lim_{\kappa\to 0}
        \tilde F_k\bigl(\alpha_{k,\ell}(\kappa)\bigr)\\
      &= \lim_{\kappa\to 0}
        \biggl[\frac{\kappa}{\alpha_{k,\ell}(\kappa)}
        \frac{I_{k+1}\bigl(\frac{\kappa}{\alpha_{k,\ell}(\kappa)}\bigr)}{
        I_k\bigl(\frac{\kappa}{\alpha_{k,\ell}(\kappa)}\bigr)}
        + \alpha_{k,\ell}(\kappa)
        \frac{J_{k+1}(\alpha_{k,\ell}(\kappa))}{
        J_k(\alpha_{k,\ell}(\kappa))}\biggr] \\
      &= \alpha_0 \frac{J_{k+1}(\alpha_0)}{J_k(\alpha_0)}.    
  \end{align*}
  This implies that $\alpha_0= j_{k+1,\ell}$.
  \medbreak

  On the other hand, Remark \ref{rem g_k} and $\alpha_{k,\ell} >
  j_{k,\ell}$ imply
  \begin{equation*}
    \lim_{\kappa\to\infty}
    G_k\Bigl(\frac{\kappa}{\alpha_{k,\ell}(\kappa)}\Bigr) = +\infty.    
  \end{equation*}
  Given that $0 = \tilde F_k(\alpha_{k,\ell})
  = G_k({\kappa}/{\alpha_{k,\ell}})- H_k(\alpha_{k,\ell})$,
  one has
  \begin{equation*}
    \lim_{\kappa\to\infty} H_k(\alpha_{k,\ell}) = +\infty  
  \end{equation*}
  and hence $J_k(\alpha_{\infty})=0$. This shows that
  $\alpha_{\infty}=j_{k,\ell}$ and concludes the proof.
\end{proof}

\begin{remark}
  Computer generated graphs of the first $\alpha_{k,\ell}$ are drawn
  on Fig.~\ref{fig:graph-alpha}.  From this, one would for instance
  naturally conjecture\footnote{Some of these inequalities are obvious
    from the bounds on $\alpha_{k,\ell}$ that were obtained in
    Theorem~\ref{thm:alpha kl}.} that
  \begin{equation*}
    \forall \kappa > 0,\quad
    \alpha_{0,1} (\kappa) < \alpha_{1,1} (\kappa)
    < \alpha_{2,1} (\kappa) < \alpha_{0,2} (\kappa)
    < \alpha_{1,2} (\kappa) < \alpha_{2,2} (\kappa).
  \end{equation*}
  However, the picture does not stay that simple once
  $\alpha_{k,\ell}$ with higher indices are drawn: many crossings
  appear.  As an example, we have drawn graph of $\alpha_{5,1}$ which
  crosses the one of $\alpha_{0,3}$ for $\kappa \approx 44.616$,
  thereby generating an eigenspace of higher dimension.
  Nevertheless we can prove the following result concerning the first
  eigenvalue.
\end{remark}

\begin{figure}[htb]
  \centering
  \newcommand{\alphamin}{1}%
  \begin{tikzpicture}[x=2pt, y=14pt]
    \draw[->] (0,\alphamin) -- (103, \alphamin) node[below]{$\kappa$};
    \draw[->] (0,\alphamin) -- (0, 12.5) node[left]{$\alpha$};
    \begin{scope}[color=gray!40]
      \draw[fill] (44.616, 9.79106) circle(1.5pt);
      \draw[<-] (46, 9.9) -- ++(10, 0.4)
      node[right]{$\alpha_{0,3} = \alpha_{5,1}$};
    \end{scope}    
    \foreach \k/\l/\xy/\c in {
      2/1/{50,5.2}/darkorange, 2/2/{50,9}/darkorange,
      0/1/{20,2.3}/darkred, 0/2/{20,6.2}/darkred, 0/3/{20, 10.4}/darkred,
      1/1/{20,4.1}/darkblue, 1/2/{20,7.8}/darkblue, 1/3/{20,11.9}/darkblue,
      5/1/{20,9.4}/darkgreen}{
      \draw[color=\c] plot file{\graphpath{alpha\k-\l.dat}};
      \node[color=\c] at (\xy) {$\alpha_{\k,\l}$};
    }
    \foreach \kl/\y\yshift in {{0,1}/2.4048256/0, {1,1}/3.8317/0,
      {0,2}/5.520078/2pt,  {1,2}/7.0155867/0,  {0,3}/8.6537279/3pt,
      {1,3}/10.173468/0, {2,1}/5.1356223/-2pt,  {2,2}/8.4172441/-3pt,
      {2,3}/11.619841/0}{
      \draw (3pt, \y) -- (-3pt, \y)
      node[left,yshift=\yshift]{\scriptsize $j_{\kl}$};
    }
  \end{tikzpicture}
  \vspace{-2ex}
  \caption{Graph of some $\alpha_{k,\ell}$.}
  \label{fig:graph-alpha}
\end{figure}
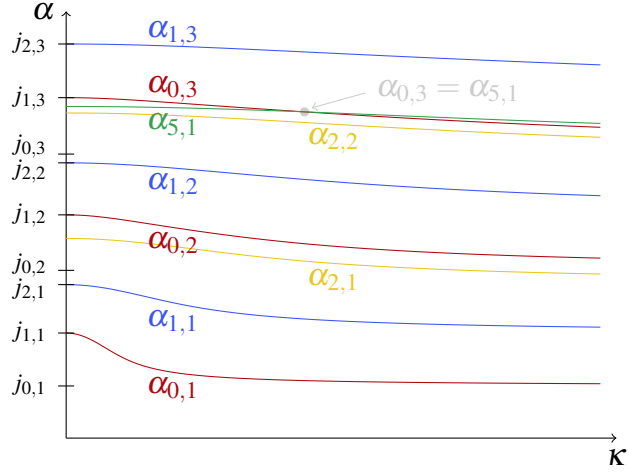

\begin{theorem}
  \label{thm:lambda1}
  The first eigenvalue of \eqref{serge5} is simple and given by
  \begin{equation}
    \label{lambda_1}
    \lambda_1 = \alpha_{0,1}^2 -
    \frac{\kappa^2}{\alpha_{0,1}^2}
  \end{equation}
  with corresponding eigenfunctions given by
  \begin{equation}
    \label{eq:R01}
    R_{0,1}(r)
    = c\, J_0(\alpha_{0,1}\, r)
    + d\, I_0\Bigl(\frac{\kappa}{\alpha_{0,1}} \, r \Bigr),
  \end{equation}
  where $(c, d)$ is a nontrivial solution to~\eqref{eq:R bd cond} with
  $\alpha=\alpha_{0,\ell}$.  Moreover $\lambda_1$ is a decreasing
  function of $\kappa$,
  \begin{equation*}
    \lim_{\kappa\to0} \lambda_{1}(\kappa) = j_{1,1}^2
    \qquad\text{and}\qquad
    \lim_{\kappa\to\infty} \lambda_{1}(\kappa) = -\infty.
  \end{equation*}
\end{theorem}

\begin{proof}
  First note that the function
  $\intervaloo{0,+\infty} \to \IR : \alpha \mapsto \alpha^2 -
  \kappa^2/\alpha^2$
  is increasing.  The localization of  $\alpha_{k,\ell}$ given in
  Theorem~\ref{thm:alpha kl} readily shows that $\alpha_{0,1}$ is the
  smaller of all $\alpha_{k,\ell}$.
  Formula \eqref{lambda_1} and \eqref{eq:R01} immediately follow.
  The monotonicity comes from Lemma~\ref{alpha decrease} and the
  asymptotic values from Lemma~\ref{kappa->0}.
\end{proof}

\section{Nodal properties of eigenfunctions}
\label{sec:nodal prop}

In this section, we prove that the eigenfunction $R_{0,1}$ is positive
and decreasing, and give precise statements about the change of sign
of the other eigenfunctions.

\begin{theorem}
  \label{R01}
  Let $ \kappa>0$ and $R_{0,1}$ be the function associated with
  $\alpha_{0,1}$ defined by~\eqref{eq:R01} in
  Theorem~\ref{thm:lambda1}.  Then $r \mapsto \abs{R_{0,1}(r)}$
  is positive in $\intervalco{0,1}$ and decreasing.
\end{theorem}

\begin{proof}
  As $R_{0,1}$  is given by \eqref{eq:R01}
  where  $(c,d)$ is a nontrivial solution of \eqref{eq:R bd cond}
  with $\alpha = \alpha_{0,1}$, we may choose
  $c= I_0\bigl(\frac{\kappa}{\alpha_{0,1}}\bigr)$ and
  $d=-J_0(\alpha_{0,1})$.
  Since $\alpha_{0,1}\in \intervaloo{j_{0,1}, j_{1,1}}$, we deduce
  that $c$ and $d$ are positive.

  Now we want to show that $v(r) := \partial_rR_{0,1}(r) <0$ for all $r \in
  \intervaloo{0,1}$. As $R_{0,1}(1)=0$, we then obtain also $R_{0,1}>0$
  on~$\intervalco{0,1}$.
  First observe that $v$ is given by
  \begin{equation}
    \label{eq:d R01}
    v(r)
    = -c\, \alpha_{0,1} J_1(\alpha_{0,1} r)
    + d \, \frac{\kappa}{\alpha_{0,1}}
    I_1\Bigl(\frac{\kappa}{\alpha_{0,1}} r\Bigr).
  \end{equation}
  A simple computation using \eqref{A:6} and  \eqref{A:6I} shows that $v$ solves
  \begin{equation}
    \label{partial R}
    \begin{array}{c}
      -\partial_r^2 v- \frac{1}{r} \partial_r v
      + \Bigl(\frac{1}{r^2} + \frac{\kappa^2}{\alpha_{0,1}^2}\Bigr) v
      = -c \alpha_{0,1} \Bigl( \alpha_{0,1}^2
      +\frac{\kappa^2}{\alpha_{0,1}^2} \Bigr)
      J_1(\alpha_{0,1} r),
      \\[3\jot]
      v(0)=0,\quad v(1)=0.
    \end{array}
  \end{equation}
  This problem can be rewritten under the form
  \begin{equation}
    \label{Laplacian partial R}
    \begin{array}{c}
      -\partial_r(r \, \partial_rv)
      + r\Bigl(\frac{1}{r^2}+\frac{\kappa^2}{\alpha_{0,1}^2}\Bigr) v
      = -c \alpha_{0,1} \Bigl( \alpha_{0,1}^2
      +\frac{\kappa^2}{\alpha_{0,1}^2} \Bigr) r\,
      J_1(\alpha_{0,1} r),
      \\[3\jot]
      v(0)=0,\quad v(1)=0.
    \end{array}
  \end{equation}
  Since
  $-c \alpha_{0,1} \Bigl( \alpha_{0,1}^2 +
  \frac{\kappa^2}{\alpha_{0,1}^2} \Bigr) r\, J_1(\alpha_{0,1} r)$
  is negative in $\intervaloo{0,1}$, multiplying~\eqref{Laplacian
    partial R} by $v^+$ and integrating we obtain
  \begin{equation*}
    0 \le \int_0^1 r \abs{\partial_r(v^+)(r)}^2 \intd r
    \le \int_0^1 -c \alpha_{0,1} \Bigl( \alpha_{0,1}^2
    + \frac{\kappa^2}{\alpha_{0,1}^2} \Bigr) r\,
    J_1(\alpha_{0,1} r) v^+ \intd r
    \le 0.    
  \end{equation*}
  This implies that $v^+\equiv 0$ i.e., $v\leq 0$. On the other hand,
  if there exists $r_0\in \intervaloo{0,1}$ such that
  $v(r_0)=\max_{[0,1]} v=0$ then $\partial_r v(r_0)=0$, $\partial_r^2 v(r_0)\leq 0$ which
  gives a contradiction with \eqref{partial R}.

  This implies that $v<0$ on $\intervaloo{0,1}$ and concludes the proof.
\end{proof}

\begin{theorem}
  \label{thm:Rk1>0}
  Let $\kappa > 0$, $k \in \IN$, and $R_{k,1}$ be the function defined
  by
  \begin{equation}
  \label{def R_k1}
    R_{k,1}(r)
    = c\,  J_k(\alpha_{k,1} r)
    + d \, I_k\Bigl(\frac{\kappa}{\alpha_{k,1}} r\Bigr)
  \end{equation}
  with $(c,d)$ a nontrivial solution to~\eqref{eq:R bd cond}.
  Then $\abs{R_{k,1}} > 0$ in $\intervalco{0,1}$.
\end{theorem}

\begin{proof}
  As  $(c,d)$ is a nontrivial solution of \eqref{eq:R bd cond} with
  $\alpha = \alpha_{k,1}$, we may choose
  $c= I_k\bigl(\frac{\kappa}{\alpha_{k,1}}\bigr)$ and
  $d=-J_k(\alpha_{k,1})$ in \eqref{def R_k1}.  Since
  $\alpha_{k,1}\in \intervaloo{j_{k,1}, j_{k+1,1}}$, we deduce that $c$
  and $d$ are positive.

  For $r \in \intervaloc{0, j_{k,1}/\alpha_{k,1}}$, $R_{k,1}(r)$ is
  clearly positive as the sum of a non-negative and a positive term.

  To prove that $R_{k,1}$ is positive on
  $\intervaloo{j_{k,1}/\alpha_{k,1}, 1}$, suppose on the contrary
  the existence of $r^* \in \intervaloo{j_{k,1}/\alpha_{k,1}, 1}$ such
  that $R_{k,1}(r^*) = 0$.  A simple computation using \eqref{A:6} and
  \eqref{A:6I} shows that $u := R_{k,1}$ is a solution to
  \begin{equation}
    \label{eq:Rk}
    \begin{array}{c}
      -\partial_r(r\, \partial_r u)
      + \Bigl( \frac{k^2}{r^2} + \frac{\kappa^2}{\alpha_{k,1}^2} \Bigr) r\, u
      = c \Bigl( \alpha_{k,1}^2 + \frac{\kappa^2}{\alpha_{k,1}^2} \Bigr)
      r \, J_k(\alpha_{k,1} r)\\[3\jot]
      u(r^*) = 0, \quad  u(1) = 0.
    \end{array}
  \end{equation}
  Note that, for any $r \in \intervaloo{r^*, 1}$,
  $\alpha_{k,1} r \in \intervaloo{j_{k,1}, \alpha_{k,1}} \subseteq
  \intervaloo{j_{k,1}, j_{k+1,1}} \subseteq \intervaloo{j_{k,1},
    j_{k,2}}$
  and so the right hand side of \eqref{eq:Rk} is negative.
  Multiplying the equation by $u^+$ and integrating
  yields
  \begin{equation*}
    0 \le \int_{r^*}^1 r \abs{\partial_r (u^+)}^2 \intd r
    \le \int_{r^*}^1 c
    \Bigl( \alpha_{k,1}^2 + \frac{\kappa^2}{\alpha_{k,1}^2} \Bigr)
    r \, J_k(\alpha_{k,1} r)
    u^+ \intd r
    \le 0,
  \end{equation*}
  and so $u^+ \equiv 0$ i.e., $u \le 0$ on $[r^*,1]$.  Now,
  evaluating~\eqref{eq:Rk} at $r = 1$ and taking into account the
  clamped boundary conditions yields
  \begin{equation*}
    \partial_r^2 R_{k,1}(1)
    = - c \Bigl( \alpha_{k,1}^2 + \frac{\kappa^2}{\alpha_{k,1}^2} \Bigr)
    J_k(\alpha_{k,1}) > 0.
  \end{equation*}
  This contradicts the fact that $R_{k,1}$ is nonpositive on $[r^*,1]$ and
  concludes the proof.
\end{proof}

\begin{lemma}
  \label{lem:alt-signs}
  Let $\kappa > 0$, $k \in \IN$, $\ell \in \IN^*$ and $R_{k,\ell}$ be
  the function defined by~\eqref{eq R} with $\alpha = \alpha_{k,\ell}$
  and $(c,d) \ne (0,0)$ a solution to~\eqref{eq:R bd cond}.
  As it is customary, let $(j'_{k,n})_{n \ge 1}$ be the positive zeros
  of $J'_k$ in increasing order, except for $k=0$ for which we set
  $j'_{0,1} = 0$.
  Then
  \begin{equation}
    \label{eq:alt-signs}
    \forall n = 1,\dotsc, \ell,\qquad
    \sign R_{k,\ell}\Bigl(\frac{j'_{k,n}}{\alpha_{k,\ell}}\Bigr)
    = \sign(c) \, (-1)^{n+1}.
  \end{equation}
\end{lemma}

\begin{proof}
  Up to a multiplicative constant, the function $R_{k,\ell}$ can be
  written as
  \begin{equation}
    \label{eq:Rkl}
    R_{k,\ell}(r)
    := c J_k(\alpha_{k,\ell} r)
    + d I_k\Bigl(\frac{\kappa}{\alpha_{k,\ell}} r \Bigr)
  \end{equation}
  with $c = I_k\bigl(\frac{\kappa}{\alpha_{k,\ell}} \bigr) > 0$ and
  $d = -J_k(\alpha_{k,\ell})$.
  To fix the ideas, the proof will be carried out for $\ell$ even; the
  case of $\ell$ odd being similar.
  Because $\alpha_{k,\ell} \in \intervaloo{j_{k,\ell}, j_{k+1,\ell}}$
  (see Theorem~\ref{thm:alpha kl}), $d < 0$.
  Note also that the lower bound on $\alpha_{k,\ell}$ implies that
  ${j'_{k,n}}/{\alpha_{k,\ell}} < 1$ for all $n = 1,\dotsc, \ell$.
  If $n$ is even, then
  $J_k(j'_{k,n}) < 0$ which immediately implies that
  $R_{k,\ell}({j'_{k,n}}/{\alpha_{k,\ell}}) < 0$.
  To conclude the proof, it remains to show that
  $R_{k,\ell}({j'_{k,n}}/{\alpha_{k,\ell}}) > 0$
  when $n$ is odd. It is well known~\cite[p.~37]{Relton}
  that $\abs{J_k(j'_{k,n})}$ decreases with respect to~$n$.
  Thus, for odd $n \in \{1,\dotsc,\ell\}$,
  $J_k(j'_{k,n}) > J_k(j'_{k,\ell+1}) \ge J_k(\alpha_{k,\ell}) = -d$
  where the last inequality results from the fact that
  $j'_{k,\ell+1}$ is the point of maximum of $J_k$ over the interval
  $\intervalcc{j_{k,\ell}, j_{k,\ell+1}}$ and $\alpha_{k,\ell} \in
  \intervaloo{j_{k,\ell}, j_{k+1,\ell}}$.
  Moreover, as $I_k$ is increasing, $I_k\Bigl(
  \frac{\kappa}{\alpha_{k,\ell}} \frac{j'_{k,n}}{\alpha_{k,\ell}} \Bigl)
  < I_k\bigl( \frac{\kappa}{\alpha_{k,\ell}} \bigr) = c$.
  Putting the last two inequalities together proves that
  $R_{k,\ell}({j'_{k,n}}/{\alpha_{k,\ell}}) > 0$ for odd~$n$.
\end{proof}

\begin{lemma}
  \label{lem:Rkl coeff0}
  Let $k \in \IN$ and $\ell \in \IN^*$.  Then
  \begin{equation*}
    \forall \kappa \in \intervaloo{0,+\infty},\qquad
    I_k\Bigl(\frac{\kappa}{\alpha_{k,\ell}(\kappa)}\Bigr)
    \alpha_{k,\ell}^k(\kappa)
    - J_k\bigl(\alpha_{k,\ell}(\kappa)\bigr)
    \Bigl( \frac{\kappa}{\alpha_{k,\ell}(\kappa)}\Bigr)^k
    > 0.
  \end{equation*}
\end{lemma}

\begin{remark}
  For odd values of $\ell$, the estimate
  $\alpha_{k,\ell} \in \intervaloo{j_{k,\ell}, j_{k+1,\ell}}$ implies
  $J_k(\alpha_{k,\ell}) < 0$ and the statement thus clearly holds.
  The following proof shows that the statement is true for all values
  of~$\ell$.
\end{remark}

\begin{proof}
  Let us rewrite the statement as
  \begin{equation*}
    g\Bigl( \frac{\kappa}{\alpha_{k,\ell}} \Bigr)
    - h(\alpha_{k,\ell}) > 0,
    \quad\text{where }
    g(z) := \frac{I_k(z)}{z^k}
    \text{ and }
    h(z) = \frac{J_k(z)}{z^k} .
  \end{equation*}
  A simple computation using~\eqref{A:4I} yields
  $\partial_z g(z) = I_{k+1}(z)/z^k > 0$.  So $g$ is increasing and,
  using the expansion \eqref{A:I(z=0)} of $I_k$ around $0$, one gets:
  \begin{equation*}
    \forall z > 0,\qquad
    g(z) > \lim_{z\to 0} g(z)
    = \frac{1}{2^k \, k!} .
  \end{equation*}
  The proof will be complete if we show
  \begin{equation}\label{eq:proph}
    \forall z > 0,\qquad
    \frac{1}{2^k \, k!}
    = \lim_{z \to 0} h(z) > h(z).
  \end{equation}
  The equality directly follows from the expansion~\eqref{A:J(z=0)}.
  Using~\eqref{A:4}, one gets $\partial_z h(z) = - J_{k+1}(z) / z^k$.
  Thus $h$ is decreasing on $\intervaloc{0, j_{k+1,1} } \supseteq
  \intervaloc{0, j_{k,1}}$.

  Because $\abs{J_k(j'_{k,n})}$ decreases with respect to  $n$
  (see~\cite[p.~37]{Relton}), one deduces that, for all $z \ge j_{k,1}$,
  $J_k(j'_{k,1}) > \abs{J_k(z)}$ and so
  $h(j'_{k,1}) > \abs{h(z)}$.  The fact that
  $j'_{k,1} \in \intervaloc{0, j_{k,1}}$ where $h$ is decreasing
  establishes the inequality~\eqref{eq:proph}.
\end{proof}

\begin{theorem}
  \label{thm:Rkl nodal}
  Let $\kappa > 0$, $k \in \IN$, $\ell \in \IN^*$ and $R_{k,\ell}$ be
  the function defined by~\eqref{eq R} with $\alpha = \alpha_{k,\ell}$
  and $(c,d) \ne (0,0)$ a solution to~\eqref{eq:R bd cond}.
  The function $R_{k,\ell}$ possesses $\ell - 1$ roots
  in $\intervaloo{0,1}$, all of which are simple.
\end{theorem}

\begin{proof}
  Lemma~\ref{lem:alt-signs} says that
  $R_{k,\ell}({j'_{k,n}}/{\alpha_{k,\ell}})$ takes alternate signs
  when $n$ runs from $1$ to $\ell$.  Thus $R_{k,\ell}$ possesses at
  least $\ell - 1$ zeros.
  It remains to show that there is only one root in each interval
  $\intervaloo{j'_{k,n}/\alpha_{k,\ell}, j'_{k,n+1}/\alpha_{k,\ell}}$,
  $n = 1,\dotsc, \ell-1$, and that there are no roots in
  $\intervaloo{0, j'_{k,1}/\alpha_{k,\ell}}$ and
  $\intervaloo{j'_{k,\ell}/\alpha_{k,\ell}, 1}$.

  A direct computation using the definition~\eqref{eq:Rkl} of
  $R_{k,\ell}$ as well as \eqref{A:6} and \eqref{A:6I} shows that $u =
  R_{k,\ell}$ is a solution to:
  \begin{equation}
    \label{eq:Rkl-edo}
    -\partial_r(r\, \partial_r u)
    + \Bigl( \frac{k^2}{r^2} + \frac{\kappa^2}{\alpha_{k,\ell}^2} \Bigr) r\, u
    = c \Bigl( \alpha_{k,\ell}^2 + \frac{\kappa^2}{\alpha_{k,\ell}^2} \Bigr)
    r \, J_k(\alpha_{k,\ell} r) .
  \end{equation}
  Note also that, for $n = 1,\dotsc, \ell$,
  \begin{equation}
    \label{eq:Rkl(jkn)}
    \sign R_{k,\ell}\Bigl( \frac{j_{k,n}}{\alpha_{k,\ell}} \Bigr)
    = \sign\biggl( d I_k\Bigl(\frac{\kappa}{\alpha_{k,\ell}} r \Bigr) \biggr)
    = \sign(d).
  \end{equation}
  The expansions \eqref{A:J(z=0)} and \eqref{A:I(z=0)} yield
  \begin{equation*}
    R_{k,\ell}(r)
    = \gamma \frac{r^k}{2^k \, k!} \bigl(1 + o(1) \bigr)
  \end{equation*}
  where $\gamma := c \alpha_{k,\ell}^k + d (\kappa/\alpha_{k,\ell})^k$
  is positive because of the choice of $c$ and $d$ in~\eqref{eq:Rkl}
  and Lemma~\ref{lem:Rkl coeff0}.  Thus, there exists $\epsilon > 0$
  such that $R_{k,\ell}(r) > 0$ for all $r \in \intervaloc{0,
    \epsilon}$.  Without loss of generality, we can assume that
  $\epsilon < j'_{k,1} / \alpha_{k,\ell}$.

  On $\intervaloo{\epsilon, j_{k,1}/\alpha_{k,\ell}}$, the right hand
  side of \eqref{eq:Rkl-edo} is positive.  Because
  $R_{k,\ell}(\epsilon)$ and $R_{k,\ell}(j'_{k,1}/\alpha_{k,\ell})$
  are both positive, the maximum principle implies that
  $R_{k,\ell} > 0$ on
  $\intervalcc{\epsilon, j'_{k,1}/\alpha_{k,\ell}}$.  Thus
  $R_{k,\ell}$ has no root in $\intervaloo{0,
    j'_{k,1}/\alpha_{k,\ell}}$.

  If $d > 0$, by \eqref{eq:Rkl(jkn)}, we have
  $R_{k,\ell}(j_{k,1}/\alpha_{k,\ell}) > 0$
  and the same reasoning shows that
  $R_{k,\ell} > 0$ on $\intervaloo{j'_{k,1}/\alpha_{k,\ell},
    j_{k,1}/\alpha_{k,\ell}}$.  If $d < 0$, $R_{k,\ell}$ must have a
  root in the previous interval.  That root is unique because, if $r_1
  < r_2$ were two roots, applying the maximum principle on
  $\intervaloo{j'_{k,1}/\alpha_{k,\ell}, r_2}$ would imply that
  $R_{k,\ell}(r_1) > 0$, a contradiction.  Moreover, the Hopf boundary
  Lemma implies that this root is simple.

  On the interval
  $\intervaloo{j_{k,1}/\alpha_{k,\ell}, j_{k,2}/\alpha_{k,\ell}}$, the
  right hand side of \eqref{eq:Rkl-edo} is negative.  If $d < 0$,
  $R_{k,\ell}$ is negative at both endpoints of
  $\intervaloo{j_{k,1}/\alpha_{k,\ell}, j'_{k,2}/\alpha_{k,\ell}}$
  and applying the maximum principle shows that $R_{k,\ell} < 0$ on
  the whole interval.  If $d > 0$, $R_{k,\ell}$ must have a root in
  the previous interval.  As before that root must be unique (if $r_1
  < r_2$ are two roots, apply the maximum principle on
  $\intervaloo{r_1, j'_{k,2}/\alpha_{k,\ell}}$ to get a contradiction)
  and simple.

  The above arguments show that in all cases, $R_{k,\ell}$ possesses a
  single root in
  $\intervaloo{j'_{k,1}/\alpha_{k,\ell}, j'_{k,2}/\alpha_{k,\ell}}$
  and that root is simple.  The same reasoning applies to all
  intervals
  $\intervaloo{j'_{k,n}/\alpha_{k,\ell}, j'_{k,n+1}/\alpha_{k,\ell}}$.

  To conclude, let us now prove that there is no root in the last
  interval $\intervaloo{j'_{k,\ell}/\alpha_{k,\ell}, 1}$.  Choosing
  $c > 0$ and $d$ as above, 
  by  Lemma~\ref{lem:alt-signs} for the first equality, \eqref{eq:Rkl(jkn)} 
  and the bounds on $\alpha_{k,\ell}$ of
  Theorem~\ref{thm:alpha kl} for the second one and evaluating
  \eqref{eq:Rkl-edo} at
  $r=1$ and taking into account the clamped boundary conditions for
  the last one, we obtain
\allowdisplaybreaks
  \begin{align}
    \sign R_{k,\ell}\Bigl(\frac{j'_{k,\ell}}{\alpha_{k,\ell}}\Bigr)
    &= (-1)^{\ell+1} ,
    \notag\\
    \sign  R_{k,\ell}\Bigl(\frac{j_{k,\ell}}{\alpha_{k,\ell}}\Bigr)
    &= \sign(d) = \sign \bigl(-J_k(\alpha_{k,\ell}) \bigr)
      = (-1)^{\ell+1} ,
    \notag\\
    \sign \partial_r^2 R_{k,\ell}(1)
    & = -\sign J_k(\alpha_{k,\ell})
      = (-1)^{\ell+1}.
      \label{eq:d2R}
  \end{align}

  Because the sign of the right hand side of~\eqref{eq:Rkl-edo} on
  $\intervaloo{j_{k,\ell-1}/\alpha_{k,\ell},
    j_{k,\ell}/\alpha_{k,\ell}}$
  is also $(-1)^{\ell+1}$ and that interval contains
  $\intervaloo{j'_{k,\ell}/\alpha_{k,\ell},
    j_{k,\ell}/\alpha_{k,\ell}}$,
  the maximum principle implies that $R_{k,\ell}$ has the same sign on
  the whole interval
  $\intervaloo{j'_{k,\ell}/\alpha_{k,\ell},
    j_{k,\ell}/\alpha_{k,\ell}}$.
  In particular, it has no root there.

  On $\intervaloo{j_{k,\ell}/\alpha_{k,\ell}, 1}$, the sign of the
  right hand side of~\eqref{eq:Rkl-edo} is $(-1)^\ell$.  Thus, if
  there was a root $r^*$ in that interval, the maximum principle
  applied to $\intervaloo{r^*, 1}$ would imply that $R_{k,\ell}$ has
  sign $(-1)^\ell$ over that interval.
  This contradicts~\eqref{eq:d2R} and shows that there is no root in
  that interval either.
\end{proof}

\begin{remark}
  \label{rem:sign-Rkl}%
  The function $R_{0,1}$ is pictured on Fig.~\ref{fig:eigenfun}
  for ``small'' and ``large'' values of $\kappa$.  The
  different graphs of $R_{0,2}$ illustrate the nodal properties
  proved in Theorem~\ref{thm:Rkl nodal}.
  In view of Fig.~\ref{fig:graph-alpha}, the second
  eigenfunction space is spanned by $R_{1,1}(r) \e^{\pm\i \theta}$ and
  thus necessarily changes sign due to its angular part.  However, its
  radial part $R_{1,1}$ does not change sign as established in
  Theorem~\ref{thm:Rk1>0}.  It is no longer monotone though.
\end{remark}

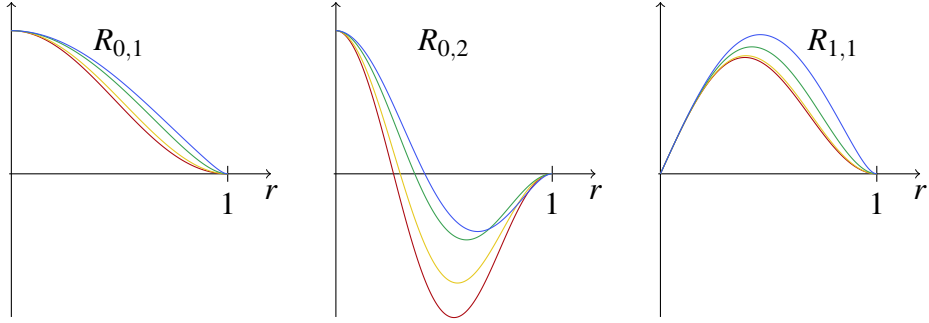
\begin{figure}[hbt]
  \centering
  \newcommand{\grapheigenfun}[2]{
    \draw[->] (-0.01, 0) -- (1.2,0) node[below]{$r$};
    \draw[->] (0, -1.) -- (0, 1.2);
    \draw (1, 3pt) -- (1, -3pt) node[below]{$1$};
    \foreach \kp/\c/\y in {0.1/darkred/0.4, 10/darkorange/0.6,
      30/darkgreen/0.8,  100/darkblue/1}{
      \draw[color=\c] plot file{\graphpath{eigenfun-#1-#2-\kp.dat}};
    }
  }
  \begin{tikzpicture}[x=15ex, y=10ex]
    \grapheigenfun{0}{1}
    \node at (0.5, 0.9) {$R_{0,1}$};
    \begin{scope}[xshift=22.5ex]
      \grapheigenfun{0}{2}
    \node at (0.5, 0.9) {$R_{0,2}$};
    \end{scope}
    \begin{scope}[xshift=45ex]
      \grapheigenfun{1}{1}
    \node at (0.8, 0.9) {$R_{1,1}$};
    \end{scope}
  \end{tikzpicture}
  \caption{Graphs of $R_{0,1}$, $R_{0,2}$ and $R_{1,1}$ for
    $\kappa= 0.1$ (red), $\kappa= 10$ (orange), $\kappa= 30$ (green)
    and $\kappa= 100$ (blue).}
  \label{fig:eigenfun}
\end{figure}

\section{Extension to any dimension}
\label{sec:N>2}

In this section, we show how the previous results may be extended to
any dimension $N \ge 2$.  This generalization is straightforward so we
only sketch the modifications to be made.

To find the eigenvalues of~\eqref{eq:pbmgeneral}, we use spherical
coordinates $r = \abs{x} \in \intervalco{0,+\infty}$ and
$\theta = \frac{x}{\abs{x}} \in \IS^{N-1}$ and the ansatz
$u = R(r) \Y_{k}(\theta)$ where $\Y_{k}$ is a spherical harmonic of
degree $k \in \IN$, i.e., an harmonic homogeneous polynomial of
degree~$k$.  Expressing $\Delta$ in spherical coordinates, equations
\eqref{serge4} and \eqref{serge4bis} become, respectively,
\begin{align}
  \label{eq:spherical-J}
  &\partial_r^2 R + \frac{N-1}{r} \partial_r R
  + \Bigl( \alpha^2 - \frac{k(k+N-2)}{r^2} \Bigr) = 0,\\
  \label{eq:spherical-I}
  &\partial_r^2 R + \frac{N-1}{r} \partial_r R
  - \Bigl( \frac{\kappa^2}{\alpha^2} + \frac{k(k+N-2)}{r^2} \Bigr) = 0  .
\end{align}
Performing the change of variables $R(r) = r^{(N-2)/2} B(r)$, one
respectively gets the following equations
\begin{align}
  \label{eq:Bessel-J}
  &\partial_r^2 B + \frac{1}{r} \partial_r B
  + \Bigl( \alpha^2 - \frac{\nu_k^2}{r^2} \Bigr) B = 0,
  \\
  \label{eq:Bessel-I}
  &\partial_r^2 B + \frac{1}{r} \partial_r B
  - \Bigl( \frac{\kappa^2}{\alpha^2} + \frac{\nu_k^2}{r^2} \Bigr) B = 0.
\end{align}
where $\nu_k := k + \frac{N-2}{2}$.  Following the same approach as
in Proposition~\ref{eigenfunctions}, one deduces that the
eigenfunctions have the form $R(r) \Y_{k}(\theta)$ with
\begin{equation}
  \label{eq:R-general}
  R(r) = c J_{\nu_k}(\alpha r)
  + d I_{\nu_k}\Bigl( \frac{\kappa}{\alpha} r \Bigr)
\end{equation}
and $\alpha$ being a positive solution to
\begin{equation}
  \label{eq:Fk-general}
  F_k(\alpha)
  := \frac{\kappa}{\alpha} J_{\nu_k}(\alpha)
  I'_{\nu_k} \Bigl(\frac{\kappa}{\alpha} \Bigr)
  - \alpha I_{\nu_k}\Bigl(\frac{\kappa}{\alpha}\Bigr)  J'_{\nu_k}(\alpha )
  =0.
\end{equation}
Because $\nu_k \ge 0$, $\nu_{k+1} = \nu_k + 1$ and the proofs do not
use the fact that $k$ is an integer, Lemma~\ref{cserge1},
Theorem~\ref{thm:alpha kl}, Lemma~\ref{alpha decrease},
Lemma~\ref{kappa->0},... remain valid with $k$ replaced by $\nu_k$.

\begin{theorem}
  \sloppy %
  Let $\kappa > 0$.  Denote $(\alpha_{k,\ell})_{\ell \ge 1}$ the
  infinitely many simple roots of $F_k$ ordered in increasing order.
  The eigenvalues of~\eqref{eq:pbm} are given by
  \begin{equation*}
    \lambda_{k,\ell}
    = \alpha_{k,\ell}^2 - \frac{\kappa^2}{\alpha_{k,\ell}^2},
    \qquad
    k \in \IN, \  \ell \in \IN^*,
  \end{equation*}
  and, in spherical coordinates
  $(r,\theta)$, the corresponding eigenfunctions are $R_{k,\ell}(r)
  \Y_k(\theta)$ where $R_{k,\ell}$ is
  defined by~\eqref{eq:R-general} with $\alpha = \alpha_{k,\ell}$ and
  $(c,d) \ne (0,0)$ chosen to satisfy the boundary conditions, and
  $\Y_k$ is a spherical harmonic of degree~$k$.
  In addition the following holds.
  \begin{itemize}
  \item The first eigenvalue is given by $\lambda_{0,1}$ and the
    corresponding eigenspace is spanned by
    $x \mapsto R_{0,1}(\abs{x})$.  Moreover
    $r \mapsto \abs{R_{0,1}(r)}$ is positive and decreasing on
    $\intervalco{0,1}$.
  \item For any $k$ and $\ell$, the function $R_{k,\ell}$ possesses
    $\ell - 1$ roots in $\intervaloo{0,1}$, all of which are simple.
  \end{itemize}
\end{theorem}

\appendix
\section{Bessel and Modified Bessel functions}

\subsection{Standard Properties}

As convenience to the reader, we gather in this section various
properties of Bessel and modified Bessel functions (see for instance
\cite{dlmf}) that are used in this paper.

\subsubsection*{Recurrence Relations and Derivatives}

The Bessel functions  $J_{\nu}$ satisfy
\allowdisplaybreaks
\begin{gather}
  \label{A:1}
  {\nu} J_{\nu}(z)=\frac{z}{2} \bigl(J_{\nu-1}(z)+J_{\nu+1}(z)\bigr),
  \\[1\jot]
  \label{A:3}
  J_{\nu}'(z)=J_{\nu-1}(z)-\frac{\nu}{z}J_{\nu}(z),
  \\[1\jot]
  \label{A:4}
  J_{\nu}'(z)=-J_{\nu+1}(z)+\frac{\nu}{z}J_{\nu}(z),
  \\[1\jot]
  \label{A:5}
  J_0'(z)=-J_{1}(z),
  \\[1\jot]
  \label{A:6}
  z^2J_{\nu}''(z)+ zJ_{\nu}'(z)+(z^2-{\nu}^2) J_{\nu}(z)=0.
\end{gather}
The modified Bessel functions  $I_{\nu}$ satisfy
\allowdisplaybreaks
\begin{gather}
  \label{A:1I}
  {\nu} I_{\nu}(z)=\frac{z}{2} \bigl(I_{\nu-1}(z)-I_{\nu+1}(z)\bigr),
  \\[1\jot]
  \label{A:3I}
  I_{\nu}'(z)=I_{\nu-1}(z)-\frac{\nu}{z}I_{\nu}(z),
  \\[1\jot]
  \label{A:4I}
  I_{\nu}'(z)=I_{\nu+1}(z)+\frac{\nu}{z}I_{\nu}(z),
  \\[1\jot]
  \label{A:5I}
  I_0'(z)=I_{1}(z),
  \\[1\jot]
  \label{A:6I}
  z^2I_{\nu}''(z)+ zI_{\nu}'(z)-(z^2+{\nu}^2) I_{\nu}(z)=0.
\end{gather}

\subsubsection*{Asymptotic behaviour}

For any given $\nu \ne -1,-2,-3,\dotsc$, when $z \to 0$,
\begin{align}
  \label{A:J(z=0)}
  J_\nu(z) &= \frac{z^\nu}{2^\nu \Gamma(\nu+1)} (1 + o(1)),
  \\[1\jot]
  \label{A:I(z=0)}
  I_\nu(z) &= \frac{z^\nu}{2^\nu \Gamma(\nu+1)} (1 + o(1)).
\end{align}

\subsection{Zeros}

When $\nu\geq0$, the positive zeros of $J_{\nu}$ are simple
and interlace according to the inequalities
\begin{equation}
  \label{A:9}
  j_{\nu,1} <j_{\nu+1,1} <j_{\nu,2} <j_{\nu+1,2} <j_{\nu,3} <\cdots
\end{equation}
On the other hand, for all $z>0$, $I_k(z)>0$ and is increasing.

\subsection{Some inequalities}

\begin{lemma}
  \label{lextLaforgia}
  For all real number $\nu> -1$, we have
  \begin{equation}
    \label{christophe}
    \forall z > 0,\qquad
    \left(\frac{ I_{\nu+1}(z)}{ I_{\nu}(z)}\right)^2
    + \frac{2\nu}{z}\frac{ I_{\nu+1}(z)}{ I_{\nu}(z)}-1
    < 0
  \end{equation}
  whence
  \begin{equation}\label{extLaforgia}
    \forall z>0,\qquad
    \frac{ I_{\nu+1}(z)}{ I_{\nu}(z)}
    < \frac{-\nu+\sqrt{\nu^2+z^2}}{z}.
  \end{equation}
\end{lemma}

\begin{proof}
  We start with the Tur\'an type inequality \cite{Baricz:10} (see also
  \cites{Lorch:94,Ismail-Muldoon:78, Laforgia-Natalini:10})
  \begin{equation}\label{turan1}
    \forall z > 0,\qquad
    I_{\nu-1}(z) I_{\nu+1}(z) < I_\nu^2(z),
  \end{equation}
  valid for all $\nu>-1$.  Then using the recurrence
  relation \eqref{A:1I} this can be rewritten under the form
  \begin{equation*}
    \Bigl(\frac{2\nu }{z}  I_\nu(z)+ I_{\nu+1}(z)\Bigr)  I_{\nu+1}(z)
    < I_\nu^2(z).    
  \end{equation*}
  Setting $u := \frac{ I_{\nu+1}(z)}{ I_{\nu}(z)}$, the previous
  inequality is equivalent to \eqref{christophe}.  The estimate
  \eqref{extLaforgia} directly follows as $z>0$.
\end{proof}

\begin{remark}
  Note that the estimate \eqref{extLaforgia} follows from
  \cite[estimate (14)]{Baricz-Ponnusamy:13} that says that for
  $\nu>-1$
  \begin{equation*}
    \forall z>0, \qquad
    \frac{I_{\nu+1}(z)}{I_{\nu}(z)}
    < \frac{-\nu-1+\sqrt{\nu^2+\frac{\nu+1}{\nu} z^2}}{
      \frac{\nu+1}{\nu} z},
  \end{equation*}
  since the right-hand side of this estimate is smaller than the
  right-hand side of~\eqref{extLaforgia}.  Actually our proof uses the
  estimate~\eqref{turan1}, while the previous estimate uses the
  converse Tur\'an type inequality
  \begin{equation*}
    \forall z > 0,\qquad
    I_\nu^2(z) - I_{\nu-1}(z) I_{\nu+1}(z)
    \le \frac{I_\nu^2(z)}{\nu +1},    
  \end{equation*}
  valid for all $\nu>-1$.
\end{remark}

\bibliography{DNT2}

\def\cprime{$'$}
\begin{bibdiv}
\begin{biblist}

\bib{Baricz:10}{article}{
      author={Baricz, {\'A}rp{\'a}d},
       title={Tur\'an type inequalities for modified {B}essel functions},
        date={2010},
        ISSN={0004-9727},
     journal={Bull. Aust. Math. Soc.},
      volume={82},
      number={2},
       pages={254\ndash 264},
         url={http://dx.doi.org/10.1017/S000497271000002X},
      review={\MR{2685149 (2011k:33009)}},
}

\bib{Baricz-Ponnusamy:13}{article}{
      author={Baricz, {\'A}rp{\'a}d},
      author={Ponnusamy, Saminathan},
       title={On {T}ur\'an type inequalities for modified {B}essel functions},
        date={2013},
        ISSN={0002-9939},
     journal={Proc. Amer. Math. Soc.},
      volume={141},
      number={2},
       pages={523\ndash 532},
         url={http://dx.doi.org/10.1090/S0002-9939-2012-11325-5},
      review={\MR{2996956}},
}

\bib{BDJM}{article}{
      author={Brown, B.~M.},
      author={Davies, E.~B.},
      author={Jimack, P.~K.},
      author={Mihajlovi{\'c}, M.~D.},
       title={A numerical investigation of the solution of a class of
  fourth-order eigenvalue problems},
        date={2000},
        ISSN={1364-5021},
     journal={R. Soc. Lond. Proc. Ser. A Math. Phys. Eng. Sci.},
      volume={456},
       pages={1505\ndash 1521},
         url={http://dx.doi.org/10.1098/rspa.2000.0573},
      review={\MR{1808762 (2001j:65162)}},
}

\bib{Chasman-these}{book}{
      author={Chasman, Laura~M.},
       title={An isoperimetric inequality for fundamental tones of free
  plates},
   publisher={ProQuest LLC, Ann Arbor, MI},
        date={2009},
        ISBN={978-1109-57141-7},
  url={http://gateway.proquest.com/openurl?url_ver=Z39.88-2004&rft_val_fmt=info:ofi/fmt:kev:mtx:dissertation&res_dat=xri:pqdiss&rft_dat=xri:pqdiss:3391901},
        note={Thesis (Ph.D.)--University of Illinois at Urbana-Champaign},
      review={\MR{2714160}},
}

\bib{Chasman}{article}{
      author={Chasman, Laura~M.},
       title={Vibrational modes of circular free plates under tension},
        date={2011},
        ISSN={0003-6811},
     journal={Appl. Anal.},
      volume={90},
      number={12},
       pages={1877\ndash 1895},
         url={http://dx.doi.org/10.1080/00036811.2010.534727},
      review={\MR{2847494}},
}

\bib{CDS}{incollection}{
      author={Coffman, C.~V.},
      author={Duffin, R.~J.},
      author={Shaffer, D.~H.},
       title={The fundamental mode of vibration of a clamped annular plate is
  not of one sign},
        date={1979},
   booktitle={Constructive approaches to mathematical models ({P}roc. {C}onf.
  in honor of {R}. {J}. {D}uffin, {P}ittsburgh, {P}a., 1978)},
   publisher={Academic Press, New York-London-Toronto, Ont.},
       pages={267\ndash 277},
      review={\MR{559499 (81g:73077)}},
}

\bib{Coffman}{article}{
      author={Coffman, Charles~V.},
       title={On the structure of solutions {$\Delta ^{2}u=\lambda u$} which
  satisfy the clamped plate conditions on a right angle},
        date={1982},
        ISSN={0036-1410},
     journal={SIAM J. Math. Anal.},
      volume={13},
      number={5},
       pages={746\ndash 757},
         url={http://dx.doi.org/10.1137/0513051},
      review={\MR{668318 (84a:35015)}},
}

\bib{CD}{article}{
      author={Coffman, Charles~V.},
      author={Duffin, Richard~J.},
       title={On the fundamental eigenfunctions of a clamped punctured disk},
        date={1992},
        ISSN={0196-8858},
     journal={Adv. in Appl. Math.},
      volume={13},
      number={2},
       pages={142\ndash 151},
         url={http://dx.doi.org/10.1016/0196-8858(92)90006-I},
      review={\MR{1162137 (93c:35113)}},
}

\bib{DNT1}{article}{
      author={De~Coster, C.},
      author={Nicaise, S.},
      author={Troestler, C.},
       title={Nodal properties of eigenfunctions of a generalized buckling
  problem on balls},
        date={2015},
     journal={Positivity},
      volume={19},
       pages={843\ndash 875},
}

\bib{dG79}{incollection}{
      author={de~Groen, P. P.~N.},
       title={Singular perturbations of spectra},
        date={1979},
   booktitle={Asymptotic analysis},
      series={Lecture Notes in Math.},
      volume={711},
   publisher={Springer, Berlin},
       pages={9\ndash 32},
}

\bib{Duffin}{article}{
      author={Duffin, R.~J.},
       title={Nodal lines of a vibrating plate},
        date={1953},
     journal={J. Math. Physics},
      volume={31},
       pages={294\ndash 299},
      review={\MR{0052293 (14,601f)}},
}

\bib{GGS}{book}{
      author={Gazzola, Filippo},
      author={Grunau, Hans-Christoph},
      author={Sweers, Guido},
       title={Polyharmonic boundary value problems. positivity preserving and
  nonlinear higher order elliptic equations in bounded domains},
      series={Lecture Notes in Mathematics},
   publisher={Springer-Verlag, Berlin},
        date={2010},
      volume={1991},
        ISBN={978-3-642-12244-6},
         url={http://dx.doi.org/10.1007/978-3-642-12245-3},
      review={\MR{2667016 (2011h:35001)}},
}

\bib{GS}{article}{
      author={Grunau, Hans-Christoph},
      author={Sweers, Guido},
       title={Positivity for perturbations of polyharmonic operators with
  {D}irichlet boundary conditions in two dimensions},
        date={1996},
        ISSN={0025-584X},
     journal={Math. Nachr.},
      volume={179},
       pages={89\ndash 102},
         url={http://dx.doi.org/10.1002/mana.19961790106},
      review={\MR{1389451 (97f:35040)}},
}

\bib{GS2}{incollection}{
      author={Grunau, Hans-Christoph},
      author={Sweers, Guido},
       title={The maximum principle and positive principal eigenfunctions for
  polyharmonic equations},
        date={1998},
   booktitle={Reaction diffusion systems ({T}rieste, 1995)},
      series={Lecture Notes in Pure and Appl. Math.},
      volume={194},
   publisher={Dekker, New York},
       pages={163\ndash 182},
      review={\MR{1472518 (98h:35050)}},
}

\bib{Hethcote:69}{article}{
      author={Hethcote, Herbert~W.},
       title={Bounds for zeros of some special functions},
        date={1970},
        ISSN={0002-9939},
     journal={Proc. Amer. Math. Soc.},
      volume={25},
       pages={72\ndash 74},
      review={\MR{0255909 (41 \#569)}},
}

\bib{Ismail-Muldoon:78}{article}{
      author={Ismail, Mourad E.~H.},
      author={Muldoon, Martin~E.},
       title={Monotonicity of the zeros of a cross-product of {B}essel
  functions},
        date={1978},
        ISSN={0036-1410},
     journal={SIAM J. Math. Anal.},
      volume={9},
      number={4},
       pages={759\ndash 767},
      review={\MR{0486686 (58 \#6388)}},
}

\bib{KLV}{article}{
      author={Kawohl, Bernhard},
      author={Levine, Howard~A.},
      author={Velte, Waldemar},
       title={Buckling eigenvalues for a clamped plate embedded in an elastic
  medium and related questions},
        date={1993},
        ISSN={0036-1410},
     journal={SIAM J. Math. Anal.},
      volume={24},
      number={2},
       pages={327\ndash 340},
         url={http://dx.doi.org/10.1137/0524022},
      review={\MR{1205530 (93j:35121)}},
}

\bib{KKM}{article}{
      author={Kozlov, V.~A.},
      author={Kondrat{\cprime}ev, V.~A.},
      author={Maz{\cprime}ya, V.~G.},
       title={On sign variability and the absence of ``strong'' zeros of
  solutions of elliptic equations},
        date={1989},
        ISSN={0373-2436},
     journal={Izv. Akad. Nauk SSSR Ser. Mat.},
      volume={53},
      number={2},
       pages={328\ndash 344},
      review={\MR{998299 (90i:35021)}},
}

\bib{Laforgia-Natalini:10}{article}{
      author={Laforgia, Andrea},
      author={Natalini, Pierpaolo},
       title={Some inequalities for modified {B}essel functions},
        date={2010},
        ISSN={1025-5834},
     journal={J. Inequal. Appl.},
       pages={Art. ID 253035, 10},
         url={http://dx.doi.org/10.1155/2010/253035},
      review={\MR{2592860 (2011b:26025)}},
}

\bib{LW:13}{article}{
      author={Lauren{\c{c}}ot, Philippe},
      author={Walker, Christoph},
       title={Sign-preserving property for some fourth-order elliptic operators
  in one dimension and radial symmetry},
        date={2014},
     journal={J. Anal. Math.},
      volume={127},
       pages={69\ndash 89},
}

\bib{Lorch:94}{article}{
      author={Lorch, Lee},
       title={Monotonicity of the zeros of a cross product of {B}essel
  functions},
        date={1994},
        ISSN={1073-2772},
     journal={Methods Appl. Anal.},
      volume={1},
      number={1},
       pages={75\ndash 80},
      review={\MR{1260384 (95b:33008)}},
}

\bib{dlmf}{misc}{
      author={of~Standards, National~Institute},
      author={Technology},
       title={Digital library of mathematical functions},
         how={http://dlmf.nist.gov/},
        date={2012},
}

\bib{Pa56}{article}{
      author={Payne, L.~E.},
       title={New isoperimetric inequalities for eigenvalues and other physical
  quantities},
        date={1956},
     journal={Comm. Pure Appl. Math.},
      volume={9},
       pages={531\ndash 542},
}

\bib{Relton}{book}{
      author={Relton, F.~E.},
       title={Applied {B}essel {F}unctions},
   publisher={Blackie \& Son Limited, London},
        date={1946},
      review={\MR{0025021}},
}

\bib{Sweers}{inproceedings}{
      author={Sweers, Guido},
       title={When is the first eigenfunction for the clamped plate equation of
  fixed sign?},
        date={2001},
   booktitle={Proceedings of the {USA}-{C}hile {W}orkshop on {N}onlinear
  {A}nalysis ({V}i\~na del {M}ar-{V}alparaiso, 2000)},
      series={Electron. J. Differ. Equ. Conf.},
      volume={6},
   publisher={Southwest Texas State Univ., San Marcos, TX},
       pages={285\ndash 296},
      review={\MR{1804781 (2001k:35070)}},
}

\bib{Wieners}{article}{
      author={Wieners, Christian},
       title={A numerical existence proof of nodal lines for the first
  eigenfunction of the plate equation},
        date={1996},
        ISSN={0003-889X},
     journal={Arch. Math. (Basel)},
      volume={66},
      number={5},
       pages={420\ndash 427},
         url={http://dx.doi.org/10.1007/BF01781561},
      review={\MR{1383907 (97c:35047)}},
}

\end{biblist}
\end{bibdiv}

\end{document}